\documentclass[11pt, a4paper]{amsart}
\usepackage{amsmath, amsthm, amsfonts, amssymb}
\input xy
\xyoption{all}

\swapnumbers

\usepackage{amsmath, amsthm, amsfonts, amssymb}
\usepackage{enumerate}
\usepackage{stackrel}
\usepackage{colordvi}
\usepackage[english]{babel}
\usepackage{color}
\usepackage{verbatim}
\usepackage[all]{xy}
\usepackage{graphicx}

\numberwithin{equation}{subsection}
\theoremstyle{plain}
  \begingroup
        \newtheorem{theorem}[equation]{Theorem}
        \newtheorem{lemma}[equation]{Lemma}
        \newtheorem{proposition}[equation]{Proposition}
        \newtheorem{corollary}[equation]{Corollary}

	    \newtheorem{definition}[equation]{Definition}

        \newtheorem{sinnadaitalica}[equation]{}

\endgroup

\theoremstyle{definition}
  \begingroup
        \newtheorem{remark}[equation]{Remark}
        
        \newtheorem{sinnadastandard}[equation]{}

\endgroup

\newcommand{\cqd}{\hfill$\Box$}  

\newcommand{\comw}{\textcolor{white}}

\newcommand{\cc}{\mathcal}

\newcommand{\ff}{\mathsf}

\newcommand{\mr}[1]{\overset {#1} {\longrightarrow}}

\newcommand{\smr}[1]{\overset {#1} {\rightarrow}}

\newcommand{\ml}[1]{\overset {#1} {\longleftarrow}}

\newcommand{\Mr}[1]{\overset {#1} {\Longrightarrow}}

\newcommand{\Ml}[1]{\overset {#1} {\Longleftarrow}}

\newcommand{\mrpairviejo}[2]
   {
    \xymatrix@C=5ex@R=2.4ex
            {
             {} \ar@<1.6ex>[r]^{#1}
	            \ar@<-1.1ex>[r]^{#2}
	         & {}
            }
   }

\newcommand{\mrpair}[2]
   {
    \xymatrix@C=5ex@R=2.4ex
            {
             {} \ar@<1ex>[r]^{#1}
	            \ar@<-1ex>[r]_{#2}
	         & {}
            }
   }

\newcommand{\mlpair}[2]
   {
    \xymatrix@C=5ex@R=2.4ex
            {
             {}
              & {} \ar@<1.0ex>[l]_{#2}
	          \ar@<-1.7ex>[l]_{#1}
            }
    }

\newcommand{\cellrd}[3] 
 {
  \xymatrix@C=7ex@R=2.4ex
         {
          {} \ar@<1.6ex>[r]^{#1}
             \ar@{}@<-1.3ex>[r]^{\!\! {#2} \, \!\Downarrow}
             \ar@<-1.1ex>[r]_{#3}
          & {}
         }
 }
 \newcommand{\scellrd}[3] 
 {
  \xymatrix@C=4.5ex@R=2.4ex
         {
          {} \ar@<1.6ex>[r]^{#1}
             \ar@{}@<-1.3ex>[r]^{\!\! {#2} \, \!\Downarrow}
             \ar@<-1.1ex>[r]_{#3}
          & {}
         }
}
 \newcommand{\modif}[3] 
 {
  \xymatrix@C=7ex@R=2.4ex
         {
          {} \ar@<1.6ex>@{=>}[r]^{#1}
             \ar@{}@<-1.3ex>@{=>}[r]^{\!\! {#2} \, \!\downarrow}
             \ar@<-1.1ex>[r]_{#3}
          & {}
         }
 }

\newcommand{\cellld}[3] 
 {
  \xymatrix@C=6ex@R=2.4ex
         {
            {}
          & {} \ar@<1.0ex>[l]^{#3}
          \ar@{}@<-1.7ex>[l]^{\!\! {#2} \, \!\Downarrow}
	                                 \ar@<-1.7ex>[l]_{#1}
         }
 }

\newcommand{\cellpairrd}[4] 
 {
  \xymatrix@C=8ex@R=2.4ex
         {
          {} \ar@<1.6ex>[r]^{#1}
             \ar@{}@<-1.3ex>[r]^{\!\! {#2} \, \!\Downarrow
                                 \;\; {#3} \, \!\Downarrow }
             \ar@<-1.1ex>[r]_{#4}
          & {}
         }
 }

\newcommand{\coLim}[2]
   {
    \underset{#1}{\underrightarrow{\ff{Lim}}}
    \; {#2}
   }

\newcommand{\Lim}[2]
   {
    \underset{#1}{\underleftarrow{\ff{Lim}}}
    \; {#2}
   }

\newcommand{\biLim}[2]
   {
    \underset{#1}{\underleftarrow{\ff{biLim}}}
    \; {#2}
   }

\newcommand{\bicoLim}[2]
   {
    \underset{#1}{\underrightarrow{\ff{biLim}}}
    \; {#2}
   }

\newcommand{\dcell}[1]  
          {
					 \ar@<8pt>@{-}[d]+<-4pt,8pt>
           \ar@<-8pt>@{-}[d]+<4pt,8pt>
           \ar@{}[d]|{#1}
          }

\newcommand{\dcellb}[1]   
          {
           \ar@<10pt>@{-}[d]+<-5pt,8pt>
           \ar@<-10pt>@{-}[d]+<5pt,8pt>
           \ar@{}[d]|{#1}
          }

\newcommand{\deq}        
         {
          \ar@{=}[d]
         }

\newcommand{\dreq}       
         {
          \ar@{=}[dr]
         }

\newcommand{\dleq}       
         {
          \ar@{=}[dl]
         }

\newcommand{\dccell}[1]    
          {
           \ar@{-}[ld]
           \ar@{-}[rd]
           \ar@{}[d]|{#1}
          }

\newcommand{\dcellbb}[1]   
          {
           \ar@<20pt>@{-}[d]+<-10pt,12pt>
           \ar@<-20pt>@{-}[d]+<10pt,12pt>
           \ar@{}[d]|{#1}
          }

\newcommand{\dl}    
          {
           \ar@<-2pt>@{-}[d]+<4pt,8pt>
          }

\newcommand{\dr}    
          {
           \ar@<2pt>@{-}[d]+<-4pt,8pt>
          }
\newcommand{\dc}[1]    
          {
           \ar@{}[d]|{#1}
          }
\newcommand{\dcr}[1]    
          {
           \ar@{}[dr]|{#1}
          }

\newcommand{\uccell}[1]      
          {
           \ar@{-}[ur]
           \ar@{}[u]|{#1}
           \ar@{-}[ul]
          }

\newcommand{\uccellb}[1]     
          {
           \ar@<-1ex>@{-}[ur]
           \ar@{}[u]|{#1}
           \ar@<1ex>@{-}[ul]
          }

\newcommand{\dcellop}[1]  
          {
					 \ar@<6pt>@{-}[d]+<6pt,8pt>
           \ar@<-6pt>@{-}[d]+<-6pt,8pt>
           \ar@{}[d]|{#1}
          }

\newcommand{\dcellopb}[1]  
          {
					 \ar@<7pt>@{-}[d]+<7pt,8pt>
           \ar@<-7pt>@{-}[d]+<-7pt,8pt>
           \ar@{}[d]|{#1}
          }

\newcommand{\did}{\ar@2{-}[d]}

\newcommand{\op}[1]
          {
           \ar@{-}[ld]
           \ar@{-}[rd]
           \ar@{}[d]|{#1}
          }

\newcommand{\cl}[1]
          {
           \ar@{-}[ur]
           \ar@{}[u]|{#1}
           \ar@{-}[ul]
          }

\newcommand{\Pro}[1]{2\hbox{-}\cc{P}ro(\cc{#1})}

\begin{document}

\title{A theory of 2-pro-objects  \\
       (with expanded proofs)
      }

\author{M. Emilia Descotte and Eduardo J. Dubuc}

\begin{abstract}
In \cite{G2}, Grothendieck develops the theory of \mbox{pro-objects} over a category $\ff{C}$.  The fundamental property of the category $\ff{Pro}(\ff{C})$ is that there is an embedding $\ff{C} \mr{c} \ff{Pro}(\ff{C})$, the category
$\ff{Pro}(\ff{C})$ is closed under small cofiltered limits, and these limits are free in the sense that for any \mbox{category} $\ff{E}$ closed under small cofiltered \mbox{limits}, pre-composition with $c$ determines an equivalence of categories
$\cc{C}at(\ff{Pro}(\ff{C}),\,\ff{E})_+ \simeq \cc{C}at(\ff{C},\, \ff{E})$, (where the
"$+$" indicates the full subcategory of the functors \mbox{preserving} cofiltered limits).  In this paper we develop a 2-dimensional theory of \mbox{pro-objects}. Given a 2-category
$\cc{C}$, we define the \mbox{2-category} $\Pro{C}$ whose objects we call 2-pro-objects. We prove that $\Pro{C}$ has all the expected basic properties adequately relativized to the \mbox{2-categorical} \mbox{setting}, including the universal property corresponding to the one \mbox{described} above. We have at hand the results of $\cc{C}at$-enriched category theory, but our theory goes beyond the \mbox{$\cc{C}at$-enriched} case since we consider the non strict notion of pseudo-limit, which is usually that of practical interest.
\end{abstract}

\maketitle

{\bf Note.} This is a version of the article \emph{A theory of 2-Pro-objects, Cahiers de topologie et g\'eom\'etrie diff\'erentielle cat\'egoriques, Vol LV, 2014}, in which we have added more details in several proofs, and utilized the elevators calculus graphical notation.
  
{\bf Introduction.} In this paper we develop a 2-dimensional theory of
\mbox{pro-objects}. Our motivation are intended applications in homotopy, in
\mbox{particular} strong shape theory. The $\check{C}$ech nerve before passing modulo
homotopy determines a 2-pro-object which is not a pro-object, leaving
outside the actual theory of pro-objects. Also, the theory of \mbox{2-pro-objects}
reveals itself a very interest subject in its own right.

Given a 2-category $\cc{C}$ , we define the 2-category $2$-$\cc{P}ro(\cc{C})$ , whose
objects we call 2-pro-objects. A 2-pro-object is a 2-functor (or diagram)
indexed in a 2-cofiltered 2-category. Our theory goes beyond enriched
category theory because in the definition of the category of morphisms of \mbox{2-pro-objects,} instead of strict
2-limits of categories, we use pseudo-limits of categories, which are usually 
those of practical interest. We prove that $2$-$\cc{P}ro(\cc{C})$  has all the expected
basic properties of the category of pro-objects, adequately relativized
to the 2-categorical setting.

Section 1 contains some background material on 2-categories. Most
of this is standard, but some results (for which we \mbox{provide} proofs) do not
appear to be in the literature. In particular we prove that \mbox{pseudolimits}
are computed pointwise in the 2-functor \mbox{2-categories} $\cc{H}om(\cc{C},\cc{D})$ and
$\cc{H}om_p(\cc{C},\cc{D})$ (\mbox{definition} \ref{ccHom}), with \mbox{2-natural} or \mbox{pseudonatural} \mbox{transformations}
as \mbox{arrows}. This result, \mbox{although} \mbox{expected}, needs \mbox{nevertheless}
a proof. We recall from \cite{DS} the \mbox{construction} of \mbox{2-filtered} pseudocolimits
of categories which is \mbox{essential} for the \mbox{computations} in the 2-category of
2-pro-objects \mbox{introduced} in \mbox{section} 2. Finally, we \mbox{consider} the notion of
flexible functors from \cite{BKP} and state a \mbox{useful} \mbox{characterization} independent
of the left \mbox{adjoint} to the inclusion $\cc{H}om(\cc{C},\cc{D}) \rightarrow \cc{H}om_p(\cc{C},\cc{D})$ (Proposition
\ref{flexiblechar}). With this \mbox{characterization} the pseudo Yoneda lemma just says that the \mbox{representable} 2-functors are
flexible. It follows also that
the \mbox{2-functor} associated to any 2-pro-object is
flexible, and this has \mbox{important}
\mbox{consequences} for a Quillen model structure in the 2-category
of \mbox{2-pro-objects} currently being developed by the authors in ongoing
research.

Section 2 contains the main results of this paper. In a first subsection
we define the 2-category of 2-pro-objects of a 2-category $\cc{C}$ and establish
the basic formula for the morphisms and 2-cells between 2-pro-objects in
terms of pseudo limits and pseudo colimits of the hom categories of $\cc{C}$.
With this, inspired in the notion of an arrow representing a morphism
of pro-objects found in \cite{AM}, in the next subsection we introduce the
notion of an arrow and a 2-cell in $\cc{C}$ representing an arrow and a 2-
cell in $2$-$\cc{P}ro(\cc{C})$ , and develop computational properties of 2-pro-objects
which are necessary in our proof that the 2-category $2$-$\cc{P}ro(\cc{C})$  is closed
under 2-cofiltered pseudo limits. In the third subsection we construct a
2-filtered category which serves as the index 2-category for the 2-filtered
pseudolimit of 2-pro-objects (Definition \ref{kequis} and Theorem \ref{teo}). This
is also inspired in a construction and proof for the same purpose found
in \cite{AM}, but which in our 2-dimensional case reveals itself very complex
and difficult to manage effectively. We were forced to have recourse to
this complicated construction because the conceptual treatment of this
problem found in \cite{G2} does not apply in the 2-dimensional case. This
is so because a 2-functor is not the pseudocolimit of 2-representables
indexed by its 2-category of elements. Finally, in the last subsection we
prove the universal properties of $2$-$\cc{P}ro(\cc{C})$ (Theorem \ref{pseudouniversal}), in a way
which is novel even if applied to the classical theory of pro-objets.

\vspace{1ex}

 {\bf Notation.} 2-Categories will be denoted with the ``mathcal" font 
 $\cc{C},\;\cc{D},\; \ldots \;$, \mbox{2-functors} with the capital ``mathff'' font, $\ff{F}$, $\ff{G}$, ... and \mbox{2-natural transformations}, pseudonatural transformations and modifications with the greek alphabet.
For objects in a 2-category, we will use capital ``mathff" font $\ff{C},\;\ff{D}, \ldots \;$, for arrows in a 2-category small case letters in ``mathff" font
$\ff{f},\;\ff{g},\;\ldots \;$, and for the 2-cells the greek \mbox{alphabet}. \mbox{However}, when a \mbox{2-category} is intended to be used as the index 2-category of a \mbox{2-diagram}, we will use small case letters $i,\;j,\;\ldots \;$ to denote its objects, and small case \mbox{letters} $u,\;v,\; \dots \;$ to denote its arrows. Categories will be denoted with capital "mathff" font.

\vspace{1ex}

Besides the usual pasting diagrams we will use the \emph{elevators calculus} for expresions denoting 2-cells. 
This is a graphic notation invented by the second author in 1969  
to write equations of natural transformations between functors. In this paper we use it for 2-cells in a 2-category. The identity arrows are left as empty spaces, the 2-cells are  
written as cells, and the identity 2-cell as a double line. Compositions are read from top to bottom. Equation \ref{basicelevator} below is the basic equality of the elevator calculus:
$$ \label{ascensor}
\xymatrix@C=0ex
         {
            \,\ff{f}'\, \dcell{\alpha'} & \,\ff{f}\, \did
          \\
            \,\ff{g}'\, \did & \,\ff{f}\, \dcell{\alpha}
          \\
             \,\ff{g}'\,  &  \,\ff{g}\, 
         }
\xymatrix@R=6ex{\\ \;\;\;=\;\;\; \\}
\xymatrix@C=0ex
         {
             \,\ff{f}'\, \did & \,\ff{f}\, \dcell{\alpha}
          \\
             \,\ff{f}'\,\dcell{\alpha'} & \,\ff{g}\, \did
          \\
             \,\ff{g}'\, & \,\ff{g}\, 
         }
\xymatrix@R=6ex{ \\ \;\;\;=\;\;\; \\}
\xymatrix@C=0ex@R=0.9ex
         {
             {} & {}
          \\
                \,\ff{f}'\, \ar@<4pt>@{-}'+<0pt,-6pt>[ddd] 
                   \ar@<-4pt>@{-}'+<0pt,-6pt>[ddd]^{\alpha'}
             &  \,\ff{f}\,  \ar@<4pt>@{-}'+<0pt,-6pt>[ddd] 
                   \ar@<-4pt>@{-}'+<0pt,-6pt>[ddd]^{\alpha}
          \\ 
             {} & {}
          \\ 
             {} & {}
          \\
             \,\ff{g}'\, & \,\ff{g}\,.
         }
$$
This allows to move cells up and down when there are no obstacles, as if they were elevators. 
\emph{With this we move cells to form configurations that fit valid equations in order to prove new equations.}

\section{Preliminaries on 2-categories}

We distinguish between \emph{small} and \emph{large} sets. For us \emph{legitimate} \mbox{categories} are categories with small hom sets, also called \emph{locally small}. We freely consider without previous warning illegitimate categories with large hom sets, for example the category of all (legitimate) categories, or functor categories with large (legitimate) exponent. They are \mbox{legitimate} as categories in some higher universe, or they can be considered as \mbox{convenient} notational \mbox{abbreviations} for large collections of data. In fact, questions of size play no overt role in this paper, except that we elect for simplicity to consider only small \mbox{2-pro-objects}. We will \mbox{explicitly} mention whether the categories are legitimate or  small when necessary. We reserve the notation $\cc{C}at$ for the legitimate 2-category of small categories, and we will denote $\cc{CAT}$ the illegitimate category (or 2-category) of all legitimate categories in some arbitrary sufficiently high universe.

 We begin with some background material on 2-categories. Most of
this is standard, but some results (for which we provide proofs) do not appear to be in the literature. We also set notation and terminology as we will explicitly use in this paper.

\subsection{Basic theory}\label{seccion2-categorias} $ $

Let $\cc{C}at$ be the category of small categories. By a \mbox{2-category}, we mean a $\cc{C}at$ enriched category. A 2-functor, a 2-fully-faithful \mbox{2-functor}, a \mbox{2-natural} transformation and a \mbox{2-equivalence} of 2-categories, are a \mbox{$\cc{C}at$-functor}, a \mbox{$\cc{C}at$-fully-faithful} functor, a $\cc{C}at$-natural transformation and a \mbox{$\cc{C}at$-equivalence} respectively.

 In the sequel we will call \emph{2-category} an structure satisfying the \mbox{following} descriptive definition free of the size restrictions implicit above. Given a \mbox{2-category,}
as usual, we denote horizontal composition by \mbox{juxtaposition}, and vertical
composition by a "$\circ$".
\begin{sinnadaitalica} {\bf 2-Category.}
\addtocounter{equation}{-1}
A \emph{2-category} $\cc{C}$ consists on objects or 0-cells \mbox{$\ff{C}$, $\ff{D}$, ...,} arrows or 1-cells $\ff{f}$, $\ff{g}$, ...,  and 2-cells $\alpha$, $\beta$, ... .

$$\ff{C} \cellrd{\ff{f}}{\alpha}{\ff{g}} \ff{D}$$

The objects and the arrows form a category (called the \mbox{underlying} \mbox{category} of $\cc{C}$), with composition (called "horizontal") denoted by \mbox{juxtaposition}. For a fixed $\ff{C}$ and $\ff{D}$, the arrows between them and the 2-cells between these arrows form a category $\cc{C}(\ff{C},\ff{D})$ under "vertical" \mbox{composition}, denoted by a "$\circ$". There is also an associative horizontal \mbox{composition} between 2-cells denoted by juxtaposition, with units
$id_{id_{\ff{C}}}$. The following is the basic \mbox{2-category} diagram: $$
\xymatrix@R0.5ex
        {
           {\;\;} \ar[rr]^{\ff{f}}
         & {\;\;}
         & {\;\;} \ar[rr]^{\ff{f}'}
         & {\;\;}
         & {\;\;}
         \\
           {\;\;}
         & {\Downarrow}\alpha
         & {\;\;}
         & {\Downarrow}{\alpha'}
         & {\;\;}
         & {\;\;}
         \\
           {\ff{C}} \ar[rr]^{\ff{g}}
         & {\;\;}
         & {\ff{D}} \ar[rr]^{\ff{g}'}
         & {\;\;}
         & {\ff{E}}
         \\
           {\;\;}
         & {\Downarrow}\beta
         & {\;\;}
         & {\Downarrow}{\beta'}
         & {\;\;}
         & {\;\;}
         \\
           {\;\;} \ar[rr]^{\ff{h}}
         & {\;\;}
         & {\;\;} \ar[rr]^{\ff{h}'}
         & {\;\;}
         & {\;\;}}
$$
with the equations
$
(\beta'\beta)\circ(\alpha' \alpha) =
           (\beta'\circ\alpha')(\beta\circ\alpha),
$
$id_{\ff{f}'}id_\ff{f}=id_{\ff{f}'\ff{f}}$.
\end{sinnadaitalica}

In particular it follows that given 
$\ff{C} \cellrd{\ff{f}}{\alpha}{\ff{g}} \ff{D}
\cellrd{\ff{f}'}{\alpha'}{\ff{g}'} \ff{E}$, we have:
\begin{equation} \label{basicelevator}
 (\alpha'\,id_{\ff{g}}) \circ (id_{\ff{f}'}\,\alpha) =
(id_{\ff{g}'}\,\alpha) \circ (\alpha'\,id_{\ff{f}}) =
(\alpha'\alpha).
\end{equation}

\vspace{1ex}

We consider juxtaposition more binding than "$\circ$",
thus $\alpha \beta \circ \gamma $ means $(\alpha \beta ) \circ \gamma$. We will abuse notation by writing $\ff{f}$ instead of $id_\ff{f}$ for morphisms $\ff{f}$ and $\ff{C}$ instead of $id_\ff{C}$ for objects $\ff{C}$.

\begin{sinnadaitalica} {\bf Dual 2-Category.} \label{oposite}
If $\cc{C}$ is a 2-category, we denote by $\cc{C}^{op}$ the \mbox{2-category} with the same objects as $\cc{C}$ but with $\cc{C}^{op}(\ff{C},\ff{D})=\cc{C}(\ff{D},\ff{C})$, i.e. we reverse the 1-cells but not the 2-cells.
\end{sinnadaitalica}

\begin{sinnadaitalica} {\bf 2-functor.}\label{2functor}
A \emph{2-functor} $\ff{F}:\cc{C} \mr{} \cc{D}$ between 2-categories is an enriched functor over $\cc{C}at$. As such, sends objects to objects, arrows to arrows and 2-cells to 2-cells, strictly preserving all the structure. \end{sinnadaitalica}

\begin{sinnadaitalica} {\bf 2-fully-faithful.}\label{2f&f}
A 2-functor $\ff{F}:\cc{C}\mr{} \cc{D}$ is said to be \mbox{\emph{2-fully-faithful}} if $\forall \ \ff{C},\ \ff{D}\in \cc{C}$, $\ff{F}_{\ff{C},\ff{D}}:\cc{C}(\ff{C},\ff{D})\mr{} \cc{D}(\ff{F}\ff{C},\ff{F}\ff{D})$ is an isomorphism of categories.
\end{sinnadaitalica}
\begin{sinnadaitalica} {\bf Pseudonatural.} \label{pseudonatural}
A \emph{pseudonatural transformation}
 $\cc{C} \cellrd{\ff{F}}{\theta}{\ff{G}} \cc{D}$ \mbox{between} 2-functors consists in a family of morphisms $\{\ff{F}\ff{C}\mr{\theta_{\ff{C}}} \ff{G}\ff{C}\}_{\ff{C}\in \cc{C}}$ and a family of invertible 2-cells $\{\ff{G}\ff{f}\theta_{\ff{C}}\Mr{\theta_\ff{f}} \theta_{\ff{D}}\ff{F}\ff{f}\}_{\ff{C}\mr{\ff{f}} \ff{D}\in \cc{C}}$

$$\xymatrix@R=.5pc@C=.5pc{\ff{F}\ff{C}\ar[rr]^{\theta_{\ff{C}}} \ar[dd]_{\ff{F}\ff{f}} & & \ff{G}\ff{C} \ar[dd]^{\ff{G}\ff{f}} \\ & \Downarrow \theta_\ff{f} & \\ \ff{F}\ff{D} \ar[rr]_{\theta_{\ff{D}}} && \ff{G}\ff{D}}$$

\noindent satisfying the following conditions:

\vspace{1ex}

\noindent PN0: $\forall \ \ff{C}\in \cc{C}$,
$\hspace{12ex} \theta_{id_\ff{C}}=id_{\theta_\ff{C}}$

\noindent PN1: $\forall \ \ff{C}\mr{\ff{f}} \ff{D}\mr{\ff{g}} \ff{E}$,
$\hspace{3ex}  \theta_g \ff{F f} \circ \ff{G g} \,\theta_\ff{f} = \theta_{\ff{gf}}.$

$$\vcenter{\xymatrix@C=-0.2pc{
 	   \ff{G}\ff{g} \did	 
 	   & & 
 	   \ff{G}\ff{f} \dl 
 	   &&
 	   \theta_\ff{C} \ar@{}[dll]|{\theta_\ff{f}} \dr
 	   \\
 	   \ff{G}\ff{g} \dl
 	   & & 
 	   \theta_\ff{D} \ar@{}[dll]|{\theta_\ff{g}} \dr
 	   &&
 	    \ff{F}\ff{f} \did
 	   \\ 
 	   \theta_{\ff{E}} 
            & &  \ff{F}\ff{g} 
            &&
            \ff{F}\ff{f}
            }}
 \vcenter{\xymatrix@C=-.4pc{\quad = \quad \quad }}
 \vcenter
   {
    \xymatrix@C=-.2pc@R=3pc
        {&
 		 \!\!\!\!\! \ff{G}\ff{g}\ff{f} \!\!\!\!\! \dl
 		&& &
 		\theta_\ff{C} \ar@{}[dll]_{\theta_{\ff{g}\ff{f}}\!\!\!\!}  \dr
 		\\
 		&
 		 \theta_\ff{E} 
 		& &&
 		 \ff{F}\ff{g}\ff{f} 
 		}
   }\ \ \ i.e.$$
 
$$\vcenter{\xymatrix@C=.5pc@R=.5pc{\ff{F}\ff{C} \ar[dd]_{\ff{F}\ff{f}} \ar[rr]^{\theta_{\ff{C}}} && \ff{G}\ff{C} \ar[dd]^{\ff{G}\ff{f}}&& \ff{F}\ff{C} \ar[dddd]_{\ff{F}\ff{g}\ff{f}} \ar[rr]^{\theta_{\ff{C}}} && \ff{G}\ff{C} \ar[dddd]^{\ff{G}\ff{g}\ff{f}} \\ 
            & \Downarrow \theta_\ff{f} &&&&&&\\
            \ff{F}\ff{D} \ar[dd]_{\ff{F}\ff{g}} \ar[rr]^{\theta_{\ff{D}}} && \ff{G}\ff{D} \ar[dd]^{\ff{G}\ff{g}}
            & = \;\; && \Downarrow \theta_{\ff{g}\ff{f}} & \\
            & \Downarrow \theta_\ff{g} &&&&& \\
            \ff{F}\ff{E} \ar[rr]_{\theta_{\ff{E}}} && \ff{G}\ff{E} && \ff{F}\ff{E}\ar[rr]_{\theta_{\ff{E}}}  && \ff{G}\ff{E}}}$$

\noindent PN2: $\forall \ \ff{C}\cellrd{\ff{f}}{\alpha}{\ff{g}}\ff{D}$,
$\hspace{5ex} \theta_\ff{g} \circ \ff{G}\alpha \,\theta_\ff{C} = \theta_\ff{D} \ff{F}\alpha \circ \theta_\ff{f}$            

$$\vcenter{\xymatrix@C=-.4pc{
                      \ff{G}\ff{f} \dcellb{\ff{G}\alpha}   
		      &&
		      \theta_\ff{C}  \did 
		      \\
		       \ff{G}\ff{g} \dl
		       &&
		      \theta_\ff{C} \dr \ar@{}[dll]|{\theta_\ff{g}}
		      \\
		       \theta_{\ff{D}} 
		       && 
		      \ff{F}\ff{g} 
		      }}
      \vcenter{\xymatrix@C=-.4pc{\quad = \quad }}
      \vcenter{\xymatrix@C=-.4pc{
		      \ff{G}\ff{f} \dl
		      &&
		      \theta_\ff{C} \dr \ar@{}[dll]|{\theta_\ff{f}} 
		      \\
		      \theta_{\ff{D}} \did
		      &&
		      \ff{F}\ff{f} \dcellb{\ff{F}\alpha} 
		      \\ 
		      \theta_\ff{D}
		      &&
		      \ff{F}\ff{g} 
		      }}\ \ \ i.e.$$

$$\vcenter{\xymatrix@C=.5pc@R=.5pc{\ff{F}\ff{C} \ar[dd]_{\ff{F}\ff{g}} \ar[rrr]^{\theta_{\ff{C}}} 
			    &&& \ff{G}\ff{C} \ar@<-2ex>[dd]_{\ff{G}\ff{g}} \ar@<2ex>[dd]^{\ff{G}\ff{f}}
			    &&& & \ff{F}\ff{C} \ar@<-2ex>[dd]_{\ff{F}\ff{g}} \ar@<2ex>[dd]^{\ff{F}\ff{f}} \ar[rrr]^{\theta_{\ff{C}}} 
			    &&& \ff{G}\ff{C} \ar[dd]^{\ff{G}\ff{f}}\\
			    & \Downarrow \theta_\ff{g} 
			    && \stackrel{\ff{G}\alpha}\Leftarrow  
			    && = 
			    & & \stackrel{\ff{F}\alpha}\Leftarrow  
			    && \Downarrow \theta_\ff{f} 
			    & \\
			    \ff{F}\ff{D}\ar[rrr]_{\theta_{\ff{D}}}  
			    &&& \ff{G}\ff{D} 
			    &&&& \ff{F}\ff{D} \ar[rrr]_{\theta_{\ff{D}}} 
			    &&& \ff{G}\ff{D}}}$$  
\end{sinnadaitalica}

\begin{sinnadaitalica} {\bf 2-Natural. }\label{2espseudo}
A 2-natural transformation $\theta$ between 2-functors is a pseudonatural transformation such that $\theta_\ff{f}=id\ \forall \ff{f}\in \cc{C}$. Equivalently, it is a $\cc{C}at$-enriched natural transformation, that is, a natural transformation  between the functors determined by $\ff{F}$ and $\ff{G}$, such that for each 2-cell
$\ff{C}\cellrd{\ff{f}}{\alpha}{\ff{g}}\ff{D}$, the equation $\ff{G}\alpha\theta_{\ff{C}}=\theta_{\ff{D}}\ff{F}\alpha$ holds.  \cqd
\end{sinnadaitalica}

\begin{sinnadaitalica} {\bf Modification.}
Given 2-functors $\ff{F}$ and $\ff{G}$ from $\cc{C}$ to
$\cc{D}$, a \mbox{\emph{modification}}
$\ff{F} \cellrd{\theta}{\rho}{\eta} \ff{G}$ between pseudonatural transformations is a family
\mbox{$\{\theta_{\ff{C}}\Mr{\rho_{\ff{C}}}\eta_{\ff{C}}\}_{\ff{C}\in \cc{C}}$} of 2-cells of
$\cc{D}$ such that:
$$
  \forall \; \ff{C}\mr{\ff{f}} \ff{D} \in \cc{C},
\hspace{2ex}  \rho_{\ff{D}} \ff{Ff} \circ \theta_\ff{f} =
 \eta_\ff{f} \circ \ff{Gf} \rho_\ff{C}
$$

$$\vcenter{\xymatrix@C=-.4pc{
		      \ff{G}\ff{f} \dl 
		      && 
		      \theta_{\ff{C}} \dr \ar@{}[dll]|{\theta_\ff{f}} 
		      \\
		      \theta_\ff{D} \dcellb{\rho_\ff{D}}
		      && 
		      \ff{F}\ff{f} \did  
		      \\
		      \eta_\ff{D}
		      &&
		      \ff{F}\ff{f}
		      }}
      \vcenter{\xymatrix@C=-.4pc{\quad = \quad }}
      \vcenter{\xymatrix@C=-.4pc{
		      \ff{G}\ff{f} \did 
		      && 
		      \theta_\ff{C} \dcellb{\rho_\ff{C}} 
		      \\
		      \ff{G}\ff{f} \dl 
		      && 
		      \eta_\ff{C} \dr \ar@{}[dll]|{\eta_\ff{f}} 
		      \\
		      \eta_\ff{D}
		      && 
		      \ff{F}\ff{f}
		      }}\ \ \ i.e.$$

$$\vcenter{\xymatrix@R=.8pc@C=.5pc
  {
  \ff{F}\ff{C} \ar[rr]^{\theta_{\ff{C}}}   \ar[dd]_{\ff{F}\ff{f}} 
  &
  & \ff{G}\ff{C} \ar[dd]^{\ff{G}\ff{f}}
  && \ff{F}\ff{C}  \ar[dd]_{\ff{F}\ff{f}} \ar@<1.5ex>[rr]^{\theta_{\ff{C}}} \ar@<-1.5ex>[rr]_{\eta_{\ff{C}}} 
  &\Downarrow \rho_{\ff{C}}
  & \ff{G}\ff{C} \ar[dd]^{\ff{G}\ff{f}} 
  \\
  & \Downarrow \theta_\ff{f} 
  &&
  = 
  && \Downarrow \eta_\ff{f}
  & 
  \\
  \ff{F}\ff{D} \ar@<1.5ex>[rr]^{\theta_{\ff{D}}} \ar@<-1.5ex>[rr]_{\eta_{\ff{D}}}
  &\Downarrow \rho_{\ff{D}}
  & \ff{G}\ff{D} 
  && 
  \ff{F}\ff{D} \ar[rr]_{\eta_\ff{D}}
  & & 
  \ff{G}\ff{D}
  }}$$

As a particular case, we have \mbox{\emph{modifications}}  between 2-natural transformations, which are families of 2-cells as above satisfying $\rho_{\ff{D}}\ff{F}\ff{f}=\ff{G}\ff{f}\rho_{\ff{C}}$.
\end{sinnadaitalica}

\begin{sinnadaitalica} {\bf 2-Equivalence.} \label{2equivalencia}
A 2-functor $\cc{C}\mr{\ff{F}}\cc{D}$ is said to be a \emph{\mbox{2-equivalence} of \mbox{2-categories}} if there exists a 2-functor $\cc{D}\mr{\ff{G}} \cc{C}$ and invertible 2-natural \mbox{transformations} $\ff{F}\ff{G}\Mr{\alpha} id_\cc{D}$ and $\ff{G}\ff{F}\Mr{\beta} id_\cc{C}$. $\ff{G}$ is said to be a \emph{quasi-inverse} of $\ff{F}$, and it is determined up to invertible 2-natural transformations.
\end{sinnadaitalica}

\begin{proposition} [] \emph{\cite[1.11]{K1}} \label{eqsiif&f}
A 2-functor $\ff{F}:\cc{C} \mr{} \cc{D}$ is a
\mbox{2-equivalence} of 2-categories if and only if it is 2-fully-faithful and essentially surjective on objects.\cqd
\end{proposition}

\begin{sinnadastandard} \label{2CAT}
It is well known that 2-categories, 2-functors and \mbox{2-natural} transformations form a 2-category
(which actually underlies a \mbox{3-category)} that we denote $2\hbox{-}\cc{CAT}$. Horizontal composition of \mbox{2-functors} and vertical composition of 2-natural transformations are the usual ones, and the horizontal composition of 2-natural transformations is defined by:
$$
\ff{C} \cellrd{\ff{F}}{\alpha}{\ff{G}} \ff{D}
\cellrd{\ff{F}'}{\alpha'}{\ff{G}'} \ff{E}
\;,\;\;
\;(\alpha'\alpha)_\ff{C} = \alpha'_{\ff{GC}}\circ \ff{F}'(\alpha_\ff{C}) \; ( = \ff{G}'(\alpha_\ff{C}) \circ \alpha'_{\ff{FC}}).
$$
\end{sinnadastandard}

\begin{definition}\label{ccHom}
Given two 2-categories $\cc{C}$ and $\cc{D}$, we consider two \mbox{2-categories} defined as follows:

\vspace{1ex}

$\cc{H}om(\cc{C},\cc{D})$: 2-functors and 2-natural transformations.

\vspace{1ex}

$\cc{H}om_p(\cc{C},\cc{D})
$: 2-functors and pseudonatural transformations.

\vspace{1ex}

In both cases the 2-cells are the modifications. To define compositions we draw the basic 2-category diagram:
$$
\xymatrix@R0.5ex
        {
           {\;\;} \ar[rr]^{\theta}
         & {\;\;}
         & {\;\;} \ar[rr]^{\theta'}
         & {\;\;}
         & {\;\;}
         \\
           {\;\;}
         & {\Downarrow}{\rho}
         & {\;\;}
         & {\Downarrow}{\rho'}
         & {\;\;}
         & {\;\;}
         \\
           {\ff{F}} \ar[rr]^{\eta}
         & {\;\;}
         & {\ff{G}} \ar[rr]^{\eta'}
         & {\;\;}
         & {\ff{H}}
         \\
           {\;\;}
         & {\Downarrow}{\,\varepsilon}
         & {\;\;}
         & {\Downarrow}{\,\varepsilon'}
         & {\;\;}
         & {\;\;}
         \\
           {\;\;} \ar[rr]^{\mu}
         & {\;\;}
         & {\;\;} \ar[rr]^{\mu'}
         & {\;\;}
         & {\;\;}
        }
\hspace{3ex}
\xymatrix@R2.5ex
     {
      (\theta'\theta)_{\ff{C}} =
                        \theta'_{\ff{C}}\theta_{\ff{C}}
     \\
      (\rho'\rho)_{\ff{C}}=\rho'_{\ff{C}} \rho_{\ff{C}}
     \\
      (\epsilon\circ \rho)_{\ff{C}} =
                  \epsilon_{\ff{C}} \circ \rho_{\ff{C}}
     }
$$
It is straightforward to check that these definitions determine a \mbox{2-category} structure. \cqd
\end{definition}
\begin{remark}[] \cite[I,4.2.]{GRAY} \label{evaluation}
Evaluation determines a \emph{quasifunctor}
\mbox{$\cc{H}om_p(\cc{C},\cc{D}) \times \cc{C} \mr{ev} \cc{D}$} (in the sense of \cite[I,4.1.]{GRAY}, in particular, fixing a variable, it is a 2-functor in the other).
 In the strict case $\cc{H}om$, evaluation is actually a \mbox{2-bifunctor.} \cqd
\end{remark}

\begin{remark}[]\cite[I,4.2]{GRAY} \label{Homisbifunctor}
Both constructions $\cc{H}om$ and $\cc{H}om_p$ \mbox{determine} a bifunctor
\mbox{$2\hbox{-}\cc{CAT}^{op} \times 2\hbox{-}\cc{CAT} \mr{} 2\hbox{-}\cc{CAT}$.} Given 2-functors
$\cc{C}' \mr{\ff{H}_0} \cc{C}$ and $\cc{D} \mr{\ff{H}_1} \cc{D}'$, and $\ff{F}\cellrd{\theta}{\rho}{\eta}\ff{G}$  in $\cc{H}om_\epsilon(\cc{C},\cc{D})(\ff{F},\ff{G})$, the definition
$\cc{H}om_\epsilon(\ff{H}_0, \ff{H}_1)(\ff{F}\cellrd{\theta}{\rho}{\eta}\ff{G}) = \ff{H}_1\ff{F}\ff{H}_0\cellrd{\ff{H}_1\theta\ff{H}_0}{\ff{H}_1\rho\ff{H}_0}{\ff{H}_1\eta\ff{H}_0}\ff{H}_1\ff{G}\ff{H}_0$ determines a functor
$\cc{H}om_\epsilon(\cc{C}, \cc{D})(\ff{F}, \ff{G}) \mr{}
\cc{H}om_\epsilon(\cc{C}', \cc{D}')(\ff{H}_1\ff{F}\ff{H}_0, \ff{H}_1\ff{G}\ff{H}_0)$, and this assignation is bifunctorial in the variable $(\cc{C}, \cc{D})$ (here $\cc{H}om_\epsilon$ denotes either $\cc{H}om$ or $\cc{H}om_p$).
\end{remark}

If $\cc{C}$ and $\cc{D}$ are 2-categories, the product 2-category
$\cc{C} \times \cc{D}$ is constructed in the usual way, and this together with the
\mbox{2-category } $\cc{H}om(\cc{C},\cc{D})$ determine a
symmetric cartesian closed structure as follows (see \cite[chapter 2]{K1} or \mbox{\cite[I,2.3.]{GRAY}):}

\begin{proposition}  \label{cartesianclosed}  The usual definitions determine an isomorphism of
\mbox{2-categories} :

$$
\cc{H}om(\cc{C},\,\cc{H}om(\cc{D},\,\cc{A})) \mr{\cong}
\cc{H}om(\cc{C} \times \cc{D},\,\cc{A}).
$$
Composing with the symmetry
$\cc{C} \times \cc{D} \mr{\cong} \cc{D} \times \cc{C}$ yields an isomorphism:
$$
\cc{H}om(\cc{C},\,\cc{H}om(\cc{D},\,\cc{A})) \mr{\cong}
\cc{H}om(\cc{D},\,\cc{H}om(\cc{C},\,\cc{A})).
$$

\vspace{-4ex}

\cqd
\end{proposition}
We use the following notation:

{\bf Notation:}
Let $\cc{C}$ be a 2-category, $\ff{C}\in \cc{C}$ and $\ff{D}\scellrd{\ff{f}}{\alpha}{\ff{g}}\ff{E}\in \cc{C}$.

\noindent \begin{enumerate}
  \item $\ff{f}_*$:  $\;\cc{C}(\ff{C},\ff{D})\mr{\ff{f}_*} \cc{C}(\ff{C},\ff{E})$, $\ff{f}_*(\ff{h} \mr{\beta} \ff{h}') = (\ff{f}\ff{h} \mr{\ff{f}\beta} \ff{f}\ff{h}')$.
						
  \item $\ff{f}^*$: $\;\cc{C}(\ff{E},\ff{C})\mr{\ff{f}^*} \cc{C}(\ff{D},\ff{C})$, $\ff{f}^*(\ff{h} \mr{\beta} \ff{h}') = (\ff{h}\ff{f} \mr{\beta \ff{f}}\ff{h}'\ff{f})$.

  \item $\alpha_*$: $\;\ff{f}_*\Mr{\alpha_*} \ff{g}_*$,
                   $(\alpha_*)_\ff{h}=\alpha \ff{h}$.

  \item $\alpha^*$: $\;\ff{f}^*\Mr{\alpha^*} \ff{g}^*$
                   $(\alpha^*)_\ff{h}=\ff{h}\alpha $.

  \item $\cc{C} \mr{\cc{C}(\ff{C},-)} \cc{C}at$:
$\cc{C}(\ff{C},-)(\ff{D}\scellrd{\ff{f}}{\alpha}{\ff{g}}\ff{E})  \;=\;  (\cc{C}(\ff{C},\ff{D})\scellrd{\ff{f}_*}{\alpha_*}{\ff{g}_*}\cc{C}(\ff{C},\ff{E}))$.

  \item $\cc{C}^{op} \mr{\cc{C}(-,\ff{C})} \cc{C}at$:
$\cc{C}(-,\ff{C})(\ff{D}\scellrd{\ff{f}}{\alpha}{\ff{g}}\ff{E})  \;=\;  (\cc{C}(\ff{D},\ff{C})\scellrd{\ff{f}^*}{\alpha^*}{\ff{g}^*}\cc{C}(\ff{E},\ff{C}))$.
						
   \item We will also denote by $\ff{f}^*$ the 2-natural transformation from \mbox{$\cc{C}(\ff{E},-)$ to $\cc{C}(\ff{D},-)$} defined by $(\ff{f}^*)_\ff{C}=\ff{f}^*$.

   \item We will also denote by $\ff{f}_*$ the 2-natural transformation from \mbox{$\cc{C}(-,\ff{D})$ to $\cc{C}(-,\ff{E})$} defined by $(\ff{f}_*)_\ff{C}=\ff{f}_*$.

    \item We will also denote by $\alpha^*$ the modification from $\ff{f}^*$ to $\ff{g}^*$ defined by $(\alpha^*)_\ff{C}=\alpha^*$.

    \item We will also denote by $\alpha_*$ the modification from $\ff{f}_*$ to $\ff{g}_*$ defined by $(\alpha_*)_\ff{C}=\alpha_*$.
 \cqd \end{enumerate}

\begin{sinnadaitalica} \label{Y2functor}
Given a locally small 2-category $\cc{C}$, the \emph{Yoneda \mbox{2-functors}} are the following (note that each one is the other for the dual 2-category):

\vspace{1ex}

a. $\cc{C} \mr{\ff{y}^{(-)}} \cc{H}om(\cc{C},\cc{C}at)^{op}$,
$\ff{y}^{\ff{C}} = \cc{C}(\ff{C},-)$,
$\ff{y}^{\ff{f}} = \ff{f}^*$
$\ff{y}^{\alpha} = \alpha^*$.

b. $\cc{C} \mr{\ff{y}_{(-)}} \cc{H}om(\cc{C}^{op},\cc{C}at)$,
$\ff{y}_{\ff{C}} = \cc{C}(-,\ff{C})$,
$\ff{y}_{\ff{f}} = \ff{f}_*$
$\ff{y}_{\alpha} = \alpha_*$.
\end{sinnadaitalica}

 Recall the Yoneda Lemma for enriched categories over $\cc{C}at$. We consider explicitly only the case $a.$ in
\ref{Y2functor}.

\begin{proposition}[\textbf{Yoneda lemma}]\label{2Yoneda}

Given a locally small \mbox{2-category} $\cc{C}$, a 2-functor
$\ff{F}:\cc{C} \mr{} \cc{C}at$ and an object $\ff{C}\in \cc{C}$, there is an isomorphism of categories, natural in $\ff{F}$.
$$
\xymatrix@R=.5pc@C=4pc{\cc{H}om(\cc{C},\cc{C}at)(\cc{C}(\ff{C},-),\ff{F}) \comw{X} \ar[r]^<<<<<<<h & \comw{XXXXX} \ff{F}\ff{C} \comw{XXXXX} \\ }
$$

\vspace{-5ex}

$$
\xymatrix@R=.5pc@C=1pc{\comw{XXXX} \theta \ar[rr]^\rho && \eta \comw{XXXX} \ar@{|->}[rrr] &&& \comw{X} \theta_\ff{C}(id_\ff{C}) \ar[rr]^{(\rho_\ff{C})_{id_\ff{C}}} && \eta_\ff{C}(id_\ff{C})}.
$$

\end{proposition}

\vspace{1ex}

\begin{proof}
The application $h$ has an inverse
$$
\xymatrix@R=.5pc@C=4pc{\comw{} \ff{F}\ff{C} \comw{X,,} \ar[r]^>>>>>>>>{\ell} & \comw{X} \cc{H}om(\cc{C},\cc{C}at)(\cc{C}(\ff{C},-),\ff{F})}
$$

\vspace{-4ex}

$$
\xymatrix@R=.5pc@C=1pc{C  \ar[r]^f & D \ \ar@{|->}[rrr] &&& \comw{XX} \ell C \ar[rr]^{\ell f} && \ell D \comw{XX}}
$$

\noindent where $(\ell C)_\ff{D}(\ff{f}\Mr{\alpha}\ff{g})=\ff{F}\ff{f}(C)\mr{(\ff{F}\alpha)_C} \ff{F}\ff{g}(C)$ and $((\ell f)_\ff{D})_\ff{f}=\ff{F}\ff{f}(f)$.
\end{proof}
\begin{corollary} \label{yonedaff}
The Yoneda 2-functors in \ref{Y2functor} are 2-fully-faithful.
\hfill\cqd
\end{corollary}

Beyond the theory of $\cc{C}at$-enriched categories, the lemma also holds for pseudonatural transformations in the following way:
\begin{proposition}[\textbf{Pseudo Yoneda lemma}]\label{pseudoYoneda}

Given a locally small \mbox{2-category} $\cc{C}$, a 2-functor
$\ff{F}:\cc{C} \mr{} \cc{C}at$ and an object $\ff{C}\in \cc{C}$, there is an  equivalence of categories, natural in $\ff{F}$.
$$
\xymatrix@R=.5pc@C=4pc
    {
      \cc{H}om_p(\cc{C},\cc{C}at)(\cc{C}(\ff{C},-),\ff{F}) \comw{X}
                                                       \ar[r]^<<<<<<<{\tilde{h}}
    & \comw{XXXXX} \ff{F}\ff{C} \comw{XXXXX}
    \\
    }
$$

\vspace{-5ex}

$$
\xymatrix@R=.5pc@C=1pc
   {
       \comw{XXX} \theta \ar[rr]^{\rho}
    && \eta \comw{XXXXX} \ar@{|->}[rrr]
   &&& \comw{X} \theta_\ff{C}(id_\ff{C}) \ar[rr]^{(\rho_\ff{C})_{id_\ff{C}}}
    && \eta_{\ff{C}}(id_\ff{C})
   }
$$
\noindent Furthermore, the quasi-inverse $\tilde{\ell}$ is a section of $\tilde{h}$, $\tilde{h}\, \tilde{\ell} = id$.

\end{proposition}
\begin{proof}
$\tilde{h}$ and $\tilde{\ell}$ are defined as in \ref{2Yoneda}, but now $\tilde{\ell}$ is only a section quasi-inverse of $\tilde{h}$.
The details can be checked by the reader. One can found a guide in \cite{nlab} for the case of lax functors and bicategories. We refer to the arguing and the notation there: In our case, the unit $\eta$ is the equality because $\ff{F}$ is a 2-functor, and the counit $\epsilon$ is an isomorphism because $\alpha$ is pseudonatural and the unitor $r$ is the equality.
\end{proof}
\begin{corollary} \label{repflexible}
For any locally small 2-category $\cc{C}$, and $\ff{C} \in \cc{C}$,  the inclusion
$\cc{H}om(\cc{C},\cc{C}at)(\cc{C}(\ff{C},-),\ff{F}) \mr{i} \cc{H}om_p(\cc{C},\cc{C}at)(\cc{C}(\ff{C},-),\ff{F})$ has a retraction
$\alpha$, natural in $\ff{F}$, $\alpha \, i = id$,
$i \, \alpha \cong id$, which determines an equivalence of categories.
\end{corollary}
\begin{proof}
Note that $i = \tilde{\ell} \, h$, then define $\alpha = \ell \, \tilde{h}$.
\end{proof}

\begin{corollary}
The Yoneda 2-functors in \ref{Y2functor} can be considered as \mbox{2-functors} landing in the $\cc{H}om_p$ 2-functor categories. In this case, they are pseudo-fully faithful (meaning that they determine equivalences and not isomorphisms between the hom categories). \cqd
\end{corollary}

\subsection{Weak limits and colimits} $ $

\vspace{1ex}

By \emph{weak} we understand any of the several ways universal properties can be relaxed in 2-categories. Note that pseudolimits and pseudocolimits (already considered in \cite{G3}) require isomorphisms, and have many advantages over bilimits and bicolimits, which only require equivalences. Their universal properties are both stronger and more convenient to use, and they play the principal role in this paper.
The defining universal properties characterize bilimits up to equivalence and pseudolimits up to isomorphism.

\vspace{1ex}

{\bf Notation}
 We consider pseudocolimits $\coLim{i\in \cc{I}}{\ff{F}i}$, and bicolimits
$\bicoLim{i\in \cc{I}}{\ff{F}i}$, of covariant 2-functors, and its dual concepts, pseudolimits $\Lim{i\in \cc{I}}{\ff{F}i}$, and bilimits $\biLim{i\in \cc{I}}{\ff{F}i}$, of contravariant \mbox{2-functors.}

\begin{definition}\label{pseudocone}
Let $\ff{F}:\cc{I} \mr{} \cc{A}$ be a 2-functor and $\ff{A}$ an object of $\cc{A}$.
A \emph{pseudocone} for $\ff{F}$ with vertex $\ff{A}$ is a pseudonatural transformation from $\ff{F}$ to the 2-functor which is constant at $\ff{A}$, i.e. it consists in a family of morphisms of
$\cc{A}$ $\{\ff{F}i \mr{\theta_i}\ff{A}\}_{i\in \cc{I}}$
and a family of invertible 2-cells of $\cc{A}$
\mbox{$\{\theta_i \Mr{\theta_u} \theta_{j}\ff{F}u\}_{i\mr{u} j \in \cc{I}}$}
  satisfying the following equations:

\vspace{1ex}

PC0: $\forall\ i \in \cc{I}$,
$\theta_{id_{i}}=id_{\theta_{i}}$

PC1: $\forall\ i \mr{u} j \mr{v} k \in \cc{I}$,
$\hspace{3ex} \theta_v\ff{F}_u \circ \theta_u = \theta_{vu}$

$$\vcenter{
      \xymatrix@C=-.4pc{
	&& \theta_{i} \dcellopb{\theta_u}
	\\
	& 
	\theta_{j} \dcellopb{\theta_v}
	&&
	\ff{F}u \did
	\\ 
	\theta_{k}
	&& 
	\ff{F}v 
	&
	\ff{F}u
	}}
\vcenter{\xymatrix@C=-.4pc{\quad \quad = \quad \quad  }						}
\vcenter{\xymatrix@C=-.4pc{ 
	    &
	    \theta_{i} \dcellop{\theta_{vu}} 
	    \\ 
	    \theta_{k}
	    &
	    \ff{F}v 
	    & 
	    \ff{F}u 
	    }} \ i.e.$$
$$\vcenter{
  \xymatrix@R=.5pc
    {
    \ff{F}i \ar[dd]_{\ff{F}u} \ar@/^4ex/[ddrr]^{\theta_i} 
    && 
    \\
    & 
    \Downarrow \theta_u 
    & 
    \\
    \ff{F}j \ar[dd]_{\ff{F}v} \ar[rr]^{\theta_{j}} 
    && 
    \ff{C} 
    \\
    & 
    \Downarrow \theta_v 
    &  
    \\
    \ff{F}k \ar@/_4ex/[uurr]_{\theta_{k}}  
    &&  
    }}
  \vcenter{
    \xymatrix@C=-.4pc{\quad \quad = \quad \quad}} 
    \vcenter{
    \xymatrix{
      \ff{F}i \ar[d]_{\ff{F}u}\ar@/^2ex/[dr]^{\theta_i}
      & 
      \\
      \ff{F}j \ar[d]_{\ff{F}v} \ar@{}[r]|{\Downarrow \theta_{vu}} 
      & 
      \ff{C} 
      \\
      \ff{F}k \ar@/_2ex/[ur]_{\theta_{k}} 
      && 
      }}
$$

PC2: $\forall\ i\cellrd{u}{\alpha}{v} j \in \cc{I}$,
$\hspace{5ex} \theta_v = \theta_j\ff{F}\alpha \circ \theta_u$

$$\vcenter{
    \xymatrix@C=-.4pc{
	  &
	  \theta_{i} \dcellopb{\theta_v} 
	  \\
	  \theta_{j}
	  & &
	  \ff{F}v
	  }}
	  \vcenter{\xymatrix@C=-.4pc{\quad \quad = \quad \quad  }						}
\vcenter{
      \xymatrix@C=-.4pc{ 
	  &
	  \theta_i \dcellop{\theta_{u}} 
	  \\
	  \theta_{j} \did 
	  && 
	  \ff{F}u \dcellb{\ff{F}\alpha} 
	  \\ 
          \theta_{j}      
          && 
          \ff{F}v 
	  }}\ \ \ \ \ \ \ i.e.$$

$$\vcenter{
    \xymatrix@R=.5pc{
	\ff{F}i \ar[dd]_{\ff{F}v} \ar@/^4ex/[ddrr]^{\theta_i}
	&& 
	\\
	& 
	\Downarrow \theta_v 
	& 
	\\
	\ff{F}j \ar[rr]_{\theta_{j}} 
	&& 
	\ff{C} 
	}}
\vcenter{\xymatrix@C=-.4pc{\quad \quad = \quad \quad}} 
\vcenter{
    \xymatrix@R=.5pc{
	\ff{F}i \ar@<-2ex>[dd]_{\ff{F}v} \ar@<2ex>[dd]^{\ff{F}u} \ar@/^4ex/[ddrr]^{\theta_i} 
	&& 
	\\
	\Ml{\ff{F}\alpha} 
        & 
        \Downarrow \theta_u 
        & 
        \\
        \ff{F}j \ar[rr]_{\theta_{j}}  
        && 
        \ff{C} 
        \\
            }}
	  $$

A \emph{morphism of pseudocones} between $\theta$ and $\eta$ with the same vertex is a modification, i.e. a family of 2-cells of $\cc{A}$ $\{\theta_{i}\Mr{\rho_{i}} \eta_{i}\}_{i\in \cc{I}}$ satisfying the following equation:

\vspace{1ex}

PCM: $\forall\ i\stackrel{u}{\rightarrow} j \in \cc{I}$,
$\;\eta_u \circ \rho_i = \rho_j \ff{F}u \circ \theta_u$

$$\vcenter{
      \xymatrix@C=-.4pc{
	  &\theta_{i} \dcellb{\rho_i} 
	  \\
	  &
	  \eta_i \dcellopb{\eta_u}
	  \\
	  \eta_{j}
	  & &
	  \ff{F}u
	  }}
\vcenter{\xymatrix@C=-.4pc{\quad \quad = \quad \quad  }						}
\vcenter{
      \xymatrix@C=-.4pc{
	  &
	  \theta_{i} \dcellopb{\theta_{u}}
	  \\ 
	  \theta_{j} \dcellb{\rho_{j}}
	  &&
	  \ff{F}u \did
	  \\
	  \eta_{j} 
	  && 
	  \ff{F}u
	  }}\ \ \ \ \ \ i.e.$$

$$\vcenter{
      \xymatrix@R=.5pc{
	  \ff{F}i \ar[dd]_{\ff{F}u} \ar@<1.5ex>@/^5ex/[ddrr]^{\theta_i}_{\Downarrow \rho_i} \ar@<-1.5ex>@/^5ex/[ddrr]_{\eta_i} 
	  && 
	  \\
	   \ar@{}[rr]_{\Downarrow \eta_u}
	  & 
	  & 
	  \\
          \ff{F}j \ar[rr]_{\eta_{j}}  
          && 
          \ff{C} 
          }}
\vcenter{\xymatrix@C=-.4pc{\quad \quad = \quad \quad}} 
\vcenter{
      \xymatrix@R=.5pc{
	  \ff{F}i \ar[dd]_{\ff{F}u} \ar@/^4ex/[rrdd]^{\theta_i} 
	  && 
	  \\ 
          \ar@{}[rr]^{\Downarrow \theta_u}
          &  
          & 
          \\
          \ff{F}j \ar@<-1.5ex>[rr]_{\eta_{j}} \ar@<1.5ex>[rr]^{\theta_{j}} \ar@{}[rr]|{\Downarrow \rho_{j}}  
          && 
          \ff{C} 
          \\
           }}
	  $$
\vspace{1ex}

Pseudocones form a category
$\ff{PC}_\cc{A}(\ff{F},\ff{A}) = \cc{H}om_p(\cc{I},\cc{A})(\ff{F}, \ff{A})$ furnished with a pseudocone $\ff{PC}_\cc{A}(\ff{F},\ff{A}) \mr{} \cc{A}(\ff{F}i,\, \ff{A})$,
$\{\theta_i\}_{i\in \cc{I}} \mapsto \theta_i$,
 for the 2-functor
\mbox{$\cc{I}^{op} \mr{\cc{A}(\ff{F}(-),\, \ff{A})} \cc{CAT}$.}
\end{definition}
\begin{remark} \label{PCbifunctor} ${ }$

Since $\cc{H}om_p(\cc{I},\cc{A})$ is a 2-category, it follows:

\noindent a. Pseudocones determine a 2-bifunctor
$\cc{H}om(\cc{I},\cc{A})^{op} \times \cc{A}
\mr{\ff{PC}_\cc{A}} \cc{CAT}$.

\vspace{1ex}

From Remark \ref{Homisbifunctor}  it follows in particular:

\noindent b. A 2-functor $\cc{A} \mr{\ff{H}} \cc{B}$ induces a functor between the categories of pseudocones
$\ff{PC}_\cc{A}(\ff{F},\ff{A}) \mr{\ff{PC}_\ff{H}} \ff{PC}_\cc{B}(\ff{HF},\ff{HA})$.
\cqd
\end{remark}
\begin{definition} \label{colimits} $ $
The \emph{pseudocolimit} in $\cc{A}$ of the 2-functor $\ff{F}$  is the universal pseudocone, denoted \mbox{$\{\ff{F}i\mr{\lambda_i} \coLim{i\in \cc{I}}{\ff{F}i}\}_{i\in \cc{I}}$,} in the sense that \mbox{$\forall\ \ff{A}\in \cc{A}$,} pre-composition with the $\lambda_i$ is an isomorphism of categories
$
\; \cc{A}(\coLim{i\in \cc{I}}{\ff{F}i},\ff{A}) \mr{\lambda^*} \ff{PC}_\cc{A}(\ff{F},\ff{A}).
$
Equivalently, there is an isomorphism of categories
\mbox{$\cc{A}(\coLim{i\in \cc{I}}{\ff{F}i},\ff{A}) \mr{\cong} \Lim{i\in \cc{I}^{op}}{\cc{A}(\ff{F}i,\ff{A})}$} commuting with the pseudocones. Remark that there is also an isomorphism of categories
\mbox{$\ff{PC}_\cc{A}(\ff{F},\ff{A}) \mr{\cong} \Lim{i\in \cc{I}^{op}}{\cc{A}(\ff{F}i,\ff{A})}$}

\vspace{1ex}

Requiring $\lambda^*$ to be an equivalence (which implies that also the other two isomorphisms above are equivalences) defines the notion of \mbox{\emph{bicolimit}.} Clearly, pseudocolimits are bicolimits.

\vspace{1ex}

We omit the explicit consideration of the dual concepts. \cqd
\end{definition}

It is well known that in the strict 2-functor 2-categories the strict limits and colimits are performed pointwise (if they exists in the codomain category). Here we establish this fact for the pseudo limits and pseudocolimits in both the strict and the pseudo 2-functor 2-categories. Abusing notation we can say that the formula
$(\coLim{i\in \cc{I}}{\ff{F}_i)(\ff{C}}) = \coLim{i\in \cc{I}}{\ff{F}_i(\ff{C}})$ holds in both 2-categories. The verification of this is straightforward but requires some care.

\begin{proposition}\label{pointwisebilimit}
Let $\cc{I}\mr{\ff{F}}\cc{A}$, $i\mapsto \ff{F}_i$ be a 2-functor where
$\cc{A}$ is either $\cc{H}om(\cc{C},\cc{D})$ or $\cc{H}om_p(\cc{C},\cc{D})$.
For each $\ff{C} \in \cc{C}$ let $\ff{F}_i\ff{C} \mr{\lambda^\ff{C}_i} \ff{L}\ff{C}$ be a pseudocolimit pseudocone in $\cc{D}$ for the 2 functor
$\cc{I} \mr{\ff{F}} \cc{A} \mr{ev(-, \ff{C})} \cc{D}$ (where $ev$ is evaluation, see \ref{evaluation}). Then $\ff{LC}$ is 2-functorial in $\ff{C}$ in such a way that $\lambda^\ff{C}_i$ becomes 2-natural and $\ff{F}_i \mr{\lambda_i} \ff{L}$ is a pseudocolimit pseudocone in
$\cc{A}$ in both cases.
By duality the same assertion holds for pseudolimits.
\end{proposition}


\begin{proof}
Given $\ff{C} \cellrd{\ff{f}}{\alpha}{\ff{g}}  \ff{D}$ in $\cc{C}$, evaluation determines a 2-cell in $\cc{H}om(\cc{I}, \cc{D})$
$
\ff{F}\ff{C} \cellrd{\ff{F}\ff{f}}{\ff{F}\alpha}{\ff{F}\ff{g}}
\ff{F}\ff{D} =
ev(\ff{F}(\,\hbox{-}\,), \ff{C} \cellrd{\ff{f}}{\alpha}{\ff{g}}  \ff{D})$.
(note that \mbox{$(\ff{FC})_i = \ff{F}_i\ff{C}$}, and similarly for $\ff{f}$, $\ff{g}$ and $\ff{\alpha}$). Then,  for each $\ff{X} \in \cc{D}$, it follows (from
%
%
\mbox{Remark \ref{PCbifunctor} a.)} that precomposing with this 2-cell determines a 2-cell (clearly 2-natural in the variable $\ff{X}$) in the right leg of the diagram below. Since the rows are isomorphisms, there is a unique 2-cell (also natural in the variable $\ff{X}$) in the left leg which makes the diagram commutative.
$$
\xymatrix@C=9.5ex
       {
        \cc{D}(\ff{LD}, \ff{X})
                     \ar[r]^<<<<<<<<<{(\lambda^\ff{D})^*}_<<<<<<<<{\cong}
                     \ar@<-2ex>[d]^{\;\Rightarrow}
                     \ar@<2ex>[d]
      & \ff{PC}_\cc{D}(\ff{F}\ff{D}, \ff{X})
                  \ar@<-2ex>[d]^{\;\Rightarrow}
                  \ar@<2ex>[d]
      \\
      \cc{D}(\ff{LC}, \ff{X})
                     \ar[r]^<<<<<<<<<{(\lambda^\ff{C})^*}_<<<<<<<<{\cong}
      & \ff{PC}_\cc{D}(\ff{F}\ff{C}, \ff{X})
      }
$$
Then, by the Yoneda lemma \ref{yonedaff}, the left leg is given by precomposing with a unique 2-cell in $\cc{D}$, that we denote
$\ff{L}\ff{C} \cellrd{\ff{L}\ff{f}}{\ff{L}\alpha}{\ff{L}\ff{g}}
                                                   \ff{L}\ff{D}$. It is clear by uniqueness that this determines a 2-functor
$\cc{C} \mr{\ff{L}} \cc{D}$.

Putting $\ff{X} = \ff{L}\ff{D}$ in the upper left corner and tracing the identity down the diagram yields the following commutative diagram of pseudocones in $\cc{D}$:
$$
\xymatrix@C=12ex@R=8ex
       {
        \ff{F}_i\ff{C}
               \ar[r]^<<<<<<<<<{\lambda_i^\ff{C}}
               \ar@<-2ex>[d]_{\ff{F}_i{\ff{f}}}
               \ar@<-2ex>[d]^<<<<<<{\ff{F}_i{\ff{\alpha}}}
                               ^<<<<<<<<<{\;\Rightarrow}
               \ar@<2ex>[d]^{\ff{F}_i{\ff{g}}}
      & \ff{L}\ff{C}
                     \ar@<-2ex>[d]_{\ff{L}{\ff{f}}}
               \ar@<-2ex>[d]^<<<<<<{\ff{L}{\ff{\alpha}}}
                               ^<<<<<<<<<{\;\Rightarrow}
                      \ar@<2ex>[d]^{\ff{L}{\ff{g}}}
     \\
       \ff{F}_i\ff{D}
                     \ar[r]^<<<<<<<<<{\lambda_i^\ff{D}}
      & \ff{L}\ff{D}
      }
$$
This shows that $\ff{L}$ is furnished with a pseudocone for $\ff{F}$ and that the $\lambda_i$ are \mbox{2-natural.} It only remains to check the universal property:

Let \mbox{$\cc{C} \mr{\ff{G}} \cc{D}$} be a 2-functor, consider the 2-functor $\cc{A} \mr{ev(-, \ff{C})} \cc{D}$. We have the following diagram, where the right leg is given by \mbox{Remark \ref{PCbifunctor} b.:}
$$
\xymatrix@C=8ex
       {
        \cc{A}(\ff{L}, \ff{G})
                     \ar[r]^<<<<<<<<<{\lambda^*}
                     \ar@<-0.5ex>[d]^{ev(-, \ff{C})}
      & \ff{PC}_\cc{A}(\ff{F}, \ff{G})
                  \ar@<-0.5ex>[d]^{\ff{PC}_{ev(-, \ff{C})}}
      \\
      \cc{D}(\ff{LC}, \ff{GC})
                     \ar[r]^<<<<<<<{(\lambda^\ff{C})^*}_<<<<<<{\cong}
      & \ff{PC}_\cc{D}(\ff{F}\ff{C}, \ff{GC})
      }
$$
We prove now that the upper row is an isomorphism. Given
$\ff{F}_i \scellrd{\theta_i}{\rho_i}{\eta_i} \ff{G}$ in
$\ff{PC}_\cc{A}(\ff{F}, \ff{G})$, it follows there exists a unique
$\ff{L}\ff{C} \cellrd{\ff{\tilde{\theta}\ff{C}}}
                     {\ff{\tilde{\rho}\ff{C}}}
                     {\ff{\tilde{\eta}\ff{C}}}
                                                   \ff{G}\ff{C}$
in $\cc{D}(\ff{LC}, \ff{GC})$ such that
$\tilde{\rho}\ff{C}\, \lambda_i^\ff{C} = \rho_i\ff{C}$.
%
%
It is necessary to show that this 2-cell actually lives in $\cc{A}$. This has to be checked for any
$\ff{C} \scellrd{\ff{f}}{\alpha}{\ff{g}}  \ff{D}$ in $\cc{C}$. In both cases it can be done considering the isomorphism
$
\xymatrix@C=6.5ex
       {
        \cc{D}(\ff{LC}, \ff{\ff{G}\ff{D}})
                     \ar[r]^<<<<<<<{(\lambda^\ff{C})^*}_<<<<<<{\cong}
        & \ff{PC}_\cc{D}(\ff{F}\ff{C}, \ff{G}\ff{D}).
      }
$
\end{proof}

We precise now what we do consider as \emph{preservation} properties of a \mbox{2-functor}. We do it in the case of pseudolimits and bilimits, but the same clearly applies to pseudocolimits and bicolimits. Let
$\cc{I}^{op} \mr{\ff{X}} \cc{C} \mr{\ff{H}} \cc{A}$ be any 2-functors.
\begin{definition} \label{preservation}
We say that $\ff{H}$ \emph{preserves} a pseudolimit (resp. bilimit) pseudocone $\ff{L} \mr{\pi_i} \ff{X}_i$ in
$\cc{C}$, if $\ff{H}\ff{L} \mr{\ff{H}\pi_i} \ff{H}\ff{X}_i$  is a pseudolimit (resp. bilimit) pseudocone in $\cc{A}$. Equivalently, if the (usual) comparison arrow is an isomorphism (resp. an equivalence) in $\cc{A}$.
\end{definition}
Note that by the very definition, the 2-representable 2-functors preserve pseudolimits and bilimits. Also, from proposition \ref{pointwisebilimit} it follows:
\begin{proposition} \label{yonedapreserves}
The Yoneda 2-functors in \ref{Y2functor} preserve pseudolimits. \flushright \cqd
\end{proposition}
Recall that small pseudolimits and pseudocolimits of locally small categories exist and are locally small, as well that the 2-category $\cc{C}at$ of small categories has all small pseudolimits and pseudocolimits \mbox{(see for example \cite{BKP}, \cite{K2}).}

\begin{sinnadastandard} \label{limincat}
We refer to the explicit construction of pseudolimits of category valued 2-functors, which is similar to the construction of pseudolimits of category-valued functors in \cite[Expos\'e VI 6.]{G3}, see full details \mbox{in \cite{mitesis}.}
\end{sinnadastandard}
It is also key to our work the explicit construction of 2-filtered \mbox{pseudocolimits}  of category valued 2-functors developed in \cite{DS}.
We recall this now.
\begin{definition} [Kennison, \cite{K}] \label{2filtered}
Let $\cc{C}$ be a 2-category. $\cc{C}$ is said to be \emph{2-filtered} if the following axioms are satisfied:

\vspace{1ex}

F0. Given two objects $\ff{C}$, $\ff{D}\in \cc{C}$, there exists an object $\ff{E} \in \cc{C}$ and arrows $\ff{C}\rightarrow \ff{E}$, $\ff{D}\rightarrow \ff{E}$.

F1. Given two arrows $\ff{C} \mrpair{\ff{f}}{\ff{g}} \ff{D}$, there exists an arrow $\ff{D}\mr{\ff{h}} \ff{E}$ and an invertible 2-cell $\alpha: \ff{h}\ff{f}\cong \ff{h}\ff{g}$.

F2. Given two 2-cells $\ff{C} \cellpairrd{}{\alpha}{\beta}{} \ff{D}$ there exists an arrow $\ff{D}\mr{\ff{h}} \ff{E}$ such that $\ff{h}\alpha=\ff{h}\beta$.

\vspace{1ex}

\noindent The dual notion of 2-cofiltered 2-category is given by the duals of axioms F0, F1 and F2.
\end{definition}

\begin{sinnadastandard} \label{defLF} {\bf Construction LL} (Dubuc-Street \cite{DS})
Let $\cc{I}$ be a 2-filtered \mbox{2-category} and $\ff{F}:\cc{I}\rightarrow \cc{C}at$ a 2-functor.
We define a category $\mathcal{L}(\ff{F})$ in two steps as follows:

\emph{\bf First step} (\cite[Definition 1.5]{DS}):

\emph{Objects}: $(C,i)$ with $C \in \ff{F}i$

\emph{Premorphisms}: A premorphism between $(C,i)$ and $(D,j)$ is a triple $(u,f,v)$ where $i \mr{u} k$, $j \mr{v} k$ in $\cc{I}$ and $\ff{F}(u)(C)\mr{f} \ff{F}(v)(D)$ in $\ff{F}k$.

\emph{Homotopies}: An homotopy between two premorphisms $(u_1,f_1,v_1)$ and $(u_2,f_2,v_2)$ is a quadruple $(w_1,w_2,\alpha,\beta)$ where $k_1\mr{w_1} k$, $k_2\mr{w_2} k$ are \mbox{1-cells} of $\mathcal{I}$ and $w_1v_1\mr{\alpha} w_2v_2$,
$w_1u_1\mr{\beta} w_2u_2$ are invertible 2-cells of $\mathcal{I}$ such that the following diagram commutes in $\ff{F}k$:
$$
\xymatrix@C=7ex
          {
           \ff{F}(w_1)\ff{F}(u_1)(C) = \ff{F}(w_1u_1)(C)
                \ar[r]^{\ff{F}(\beta)C}
                \ar@<-6ex>[d]_{\ff{F}(w_1)(f_1)}
          &
           \ff{F}(w_2u_2)(C) = \ff{F}(w_2)\ff{F}(u_2)(C)
                \ar@<6ex>[d]^{\ff{F}(w_2)(f_2)}
          \\
           \ff{F}(w_1)\ff{F}(v_1)(D) = \ff{F}(w_1v_1)(D)
                \ar[r]_{\ff{F}(\alpha)D}
          &
           \ff{F}(w_2v_2)(D) = \ff{F}(w_2)\ff{F}(v_2)(D)
          }
$$

We say that two premorphisms $f_1, f_2$ are equivalent if there is an \mbox{homotopy} between them. In that case, we write $f_1 \sim f_2$.

\vspace{1ex}

Equivalence is indeed an equivalence relation, and premorphisms can be (non uniquely) composed. Up to equivalence, composition is independent of the choice of representatives and of the choice of the composition between them. Since associativity holds and
identities exist, the following actually does define a category:

\vspace{1ex}

\emph{\bf Second step} (\cite[Definition 1.13]{DS}):

\emph{Objects}: $(C, \; i)$ with $C \in
Fi$.

\emph{Morphisms}: equivalence classes of premorphisms.

\emph{Composition}:
defined by composing representative premorphisms.
\end{sinnadastandard}

\begin{proposition}\emph{\cite[Theorem 1.19]{DS}} \label{construccionDS}
Let $\cc{I}$ be a 2-filtered \mbox{2-category,} \mbox{$\ff{F}:\cc{I}\rightarrow \cc{C}at$} a 2-functor, $i \mr{u} j$ in $\cc{I}$ and $C \mr{f} D \in \ff{F}i$.
The \mbox{following} formulas define a pseudocone $\ff{F} \Mr{\lambda} \cc{L}(\ff{F})$:
$$
\lambda_i(C) =(C,i)
\hspace{3ex}
\lambda_i(f) =[i,f,i]
\hspace{3ex}
(\lambda_u)_C =[u,\ff{F}u(C),j]
$$
which is a pseudocolimit for the 2-functor $\ff{F}$. \cqd
\end{proposition}

\subsection{Further results.} $ $

A. Joyal pointed to us the notion of \emph{flexible} functors, related with some of our results on pseudo colimits of representable 2-functors. We \mbox{recall} now this notion since it bears some significance for the concept of \mbox{2-pro-object} developed in this paper. Any 2-pro-object determines a 2-functor which is flexible, and some of our results find their right place stated in the context of flexible 2-functors.

\vspace{1ex}

{\bf Warning:} \emph{In this subsection 2-categories are assumed to be locally small, except the illegitimate constructions $\cc{H}om$ and $\cc{H}om_p$.}

The inclusion $\cc{H}om(\cc{C},\cc{C}at) \mr{i} \cc{H}om_p(\cc{C},\cc{C}at)$ has a left adjoint $(-)' \dashv i$, we refer the reader to \cite{BKP}. The 2-natural counit of this adjunction $\ff{F}' \Mr{\varepsilon_\ff{F}} \ff{F}$ is an equivalence in $\cc{H}om_p(\cc{C},\cc{C}at)$, with a section given by the pseudonatural unit $\ff{F} \Mr{\eta_\ff{F}} \ff{F'}$,
$\varepsilon_\ff{F} \eta_\ff{F} = id_\ff{F}$,
$\eta_\ff{F} \varepsilon_\ff{F} \cong id_\ff{F'}$, \mbox{\cite[Proposition 4.1]{BKP}}

\begin{definition}[\emph{\cite[Proposition 4.2]{BKP}}] \label{flexible}
A 2-functor $\cc{C} \mr{\ff{F}} \cc{C}at$ is \emph{flexible} if the counit
$\ff{F'} \Mr{\varepsilon_\ff{F}} \ff{F}$ has a 2-natural section $\ff{F} \Mr{\lambda} \ff{F'}$,
$\varepsilon_\ff{F} \lambda = id_\ff{F}$,
\mbox{$\lambda \varepsilon_\ff{F} \cong id_\ff{F'}$}, which determines an equivalence in $\cc{H}om(\cc{C},\cc{C}at)$.
\end{definition}

We state now a useful characterization of flexible 2-functors $\ff{F}$ independent of the left adjoint $(-)'$, the proof will appear elsewhere \cite{DD}.

\begin{proposition} \label{flexiblechar}
A 2-functor $\cc{C} \mr{\ff{F}} \cc{C}at$ is flexible
$\iff$ for all \mbox{2-functors} $\ff{G}$, the inclusion $\cc{H}om(\cc{C},\cc{C}at)(\ff{F},\ff{G}) \mr{i_\ff{G}} \cc{H}om_p(\cc{C},\cc{C}at)(\ff{F},\ff{G})$ has a retraction
$\alpha_\ff{G}$ natural in $\ff{G}$, $\alpha_\ff{G} i_\ff{G} = id$, $i_\ff{G} \alpha_\ff{G} \cong id$, which determines an equivalence of categories. \cqd
\end{proposition}
Let $\cc{H}om(\cc{C},\cc{C}at)_f$ and  $\cc{H}om_p(\cc{C},\cc{C}at)_f$ be the subcategories whose objects are the flexible 2-functors. We have the following corollaries:
\begin{corollary} \label{flexible2=p}
The 2-categories $\cc{H}om(\cc{C},\cc{C}at)_f$ and $\cc{H}om_p(\cc{C},\cc{C}at)_f$ are \emph{pseudoequivalent} in the sense they have the same objects and retract equivalent hom categories. \cqd
\end{corollary}
We mention that following the usual lines (based in the axiom of choice) in the proof of \ref{eqsiif&f}, it can be seen that the inclusion 2-functor
$\cc{H}om(\cc{C},\cc{C}at)_f \mr{} \cc{H}om_p(\cc{C},\cc{C}at)_f$ has the identity (on objects) as a retraction quasi-inverse \emph{pseudofunctor}, with the equality as the invertible pseudonatural transformation $\ff{F} \mr{=} \ff{F}$ in  $\cc{H}om_p(\cc{C},\cc{C}at)_f$.

An important property of flexible 2-functors, false in general, is the \mbox{following}:
\begin{corollary}\label{eqencadaCeseq}
Let $\theta:\ff{G}\Rightarrow \ff{F}\in \cc{H}om(\cc{C},\cc{C}at)_f$ be such that $\theta_\ff{C}:\ff{G}\ff{C}\rightarrow \ff{F}\ff{C}$ is an equivalence of categories for each $\ff{C}\in \cc{C}$. Then, $\theta$ is an equivalence in $\cc{H}om(\cc{C},\cc{C}at)_f$.
\end{corollary}
\begin{proof}
It is easy to check that there is a pseudonatural transformation \mbox{$\eta':\ff{F}\Rightarrow \ff{G}$} such that $\theta\eta'\cong \ff{F}$ and $\eta'\theta\cong \ff{G}$ in $\cc{H}om_p(\ff{F},\ff{F})$ and $\cc{H}om_p(\ff{G},\ff{G})$ respectively. Now, by \ref{flexiblechar}, there is a 2-natural transformation $\eta:\ff{F}\Rightarrow \ff{G}$ such that $\eta\cong \eta'$ in $\cc{H}om_p(\ff{F},\ff{G})$. Then, $\theta\eta\cong \ff{F}$ and $\eta\theta\cong \ff{G}$ in $\cc{H}om(\ff{F},\ff{F})$ and $\cc{H}om(\ff{G},\ff{G})$ respectively and so $\theta$ is an equivalence in $\cc{H}om(\cc{C},\cc{C}at)$.
\end{proof}

\begin{proposition} \label{colimflexible}
Small pseudocolimits of flexible 2-functors are \mbox{flexible.}
\end{proposition}
\begin{proof}
Let $\ff{F} = \coLim{j \in \cc{I}}{\ff{F}_j}$, where each $\ff{F}_j$ is flexible, and let $\ff{G}$ be any other 2-functor. Set $\cc{A} = \cc{H}om(\cc{C},\cc{C}at)$ and
$\cc{A}_p = \cc{H}om_p(\cc{C},\cc{C}at)$. Then:
$$
\cc{A}(\ff{F}, \ff{G}) \cong
\Lim{j \in \cc{I}}{\cc{A}(\ff{F}_j,\, \ff{G})} \mr{i}
\Lim{j \in \cc{I}}{\cc{A}_p(\ff{F}_j,\, \ff{G})} \cong
\cc{A}_p(\ff{F}, \ff{G}).
$$
The two isomorphisms are given by definition  \ref{colimits}. The arrow $i$ is the pseudolimit of the equivalences with retraction quasi-inverses corresponding to each $\ff{F}_j$. It is not difficult to check that $i$  is also such an equivalence.
\end{proof}
It follows also from \ref{flexiblechar} that the pseudo-Yoneda lemma (\ref{pseudoYoneda}, \ref{repflexible}) says that the representable 2-functors are flexible, so we have:
\begin{corollary}\label{pseudo=2forpro}
Small pseudocolimits of representable 2-functors are \mbox{flexible.} \cqd
\end{corollary}
Note that  \ref{colimflexible} and \ref{pseudo=2forpro} hold for any pseudocolimit that may exist.
\section{2-Pro-objects}\label{2-Pro-objects}

{\bf Warning:} \emph{In this section 2-categories are assumed to be locally small, except illegitimate constructions as $\cc{H}om$, $\cc{H}om_p$ or
$2\hbox{-}\cc{CAT}$.}

\vspace{1ex}

The main results of this paper are in this section. In the first subsection we define the 2-category of 2-pro-objects of a 2-category $\cc{C}$ and establish the basic formula for the morphisms and 2-cells of this \mbox{2-category.} Then in the next subsection we develop the notion of a \mbox{2-cell} in $\cc{C}$ \emph{representing} a 2-cell in $\Pro{C}$,  inspired in  the 1-dimensional notion of an arrow  representing a morphism of pro-objects found in \cite{AM}. We use this in the third subsection to construct the $2$-filtered category which serves as the index \mbox{2-category} for the 2-filtered pseudolimit of \mbox{2-pro-objects.} This is also inspired in a \mbox{construction} for the same purpose found in \cite{AM}. We were forced to have recourse to this complicated construction because the conceptual treatment of this problem found in \cite{G2} does not apply in the 2-category case. This is so because a \mbox{2-functor} is not the pseudocolimit of 2-representables indexed by its 2-category of elements. 
Finally, in the last
subsection we prove the universal properties of $\Pro{C}$.

\subsection{Definition of the 2-category of 2-pro-objects}
$ $

In this subsection we define the 2-category of 2-pro-objects of a fixed \mbox{2-category} and prove its basic properties. A 2-pro-object over a \mbox{2-category} $\cc{C}$ will be a small \mbox{2-cofiltered} diagram in $\cc{C}$ and it will be the pseudolimit of it's own diagram in the 2-category $\Pro{C}$.

\begin{definition} \label{2proc}
Let $\cc{C}$ be a 2-category. We define the 2-category of \mbox{2-pro-objects} of $\cc{C}$, which we denote by $2$-$\cc{P}ro(\cc{C})$, as follows:
\begin{enumerate}
\item
Its objects are the 2-functors
$\cc{I}^{op}\mr{\ff{X}}\cc{C}$,
$\ff{X} = (\ff{X}_i,\, \ff{X}_u,\, \ff{X}_\alpha)_{i,\, u,\, \alpha \in \cc{I}}$, with
$\cc{I}$ a small 2-filtered  2-category. Often we are going to abuse the notation by saying $\ff{X} = (\ff{X}_i)_{i\in \cc{I}}$.
\item
If $\ff{X}=(\ff{X}_i)_{i\in \cc{I}}$ and $\ff{Y}=(\ff{Y}_j)_{j\in \cc{J}}$ are two 2-pro-objects,
$$
\Pro{C}(\ff{X},\ff{Y})=\cc{H}om(\cc{C},\, \cc{C}at)^{op}(\Lim{i \in \cc{I}^{op}}{\cc{C}(\ff{X}_i,-)}, \Lim{j \in \cc{J}^{op}}{\cc{C}(\ff{Y}_j,-)})
$$
$$
=\cc{H}om(\cc{C},\, \cc{C}at)(\coLim{j \in \cc{J}}{\cc{C}(\ff{Y}_j,-)},\coLim{i \in \cc{I}}{\cc{C}(\ff{X}_i,-)})
$$
\end{enumerate}

The compositions are given by the corresponding compositions in the \mbox{2-category} $\cc{H}om(\cc{C},\, \cc{C}at)^{op}$ so it is easy to check that $2$-$\cc{P}ro(\cc{C})$ is indeed a 2-category.
\end{definition}

\begin{proposition}\label{eqconloscolimderep}
By definition there is a 2-fully-faithful \mbox{2-functor}
\mbox{$2\hbox{-}\cc{P}ro(\cc{C}) \mr{\ff{L}} \cc{H}om(\cc{C},\, \cc{C}at)^{op}$.}
Thus, there is a contravariant \mbox{2-equivalence} of 2-categories \mbox{$2$-$\cc{P}ro(\cc{C}) \mr{\ff{L}} \cc{H}om(\cc{C},\, \cc{C}at)^{op}_{fc}$}, where $\cc{H}om(\cc{C},\, \cc{C}at)_{fc}$ stands for the full subcategory of $\cc{H}om(\cc{C},\, \cc{C}at)$ whose objects are those \mbox{2-functors} which are small 2-filtered pseudocolimits of \mbox{representable} 2-functors. \mbox{However,} it is important to note that this equivalence is not injective on \mbox{objects.} \cqd
\end{proposition}
From Corollary \ref{pseudo=2forpro} it follows:
\begin{proposition} \label{proisflexible}
For any 2-pro-object $\ff{X}$, the corresponding \mbox{2-functor} \mbox{$\ff{LX}$ is flexible.} \cqd

\end{proposition}
\begin{remark}
If we use pseudonatural transformations to define morphisms of 2-pro-objects we obtain a 2-category $2$-$\cc{P}ro_p(\cc{C})$, which anyway, by \ref{proisflexible}, results pseudoequivalent  (see \ref{flexible2=p}) to
$2$-$\cc{P}ro(\cc{C})$, with the same objects and retract equivalent hom categories. We think our choice of morphisms, which is much more convenient to use, will prove to be the good one for the applications.
\end{remark}

Next we establish the basic formula which is essential in many \mbox{computations} in the $2$-category $2$-$\cc{P}ro(\cc{C})$:
\begin{proposition}\label{iso}
There is an isomorphism of categories:
\addtocounter{equation}{-1}
\begin{equation} \label{basica2}
  \Pro{C}(\ff{X},\ff{Y})\cong \Lim{j\in \cc{J}^{op}}{\coLim{i\in \cc{I}}
  {\cc{C}(\ff{X}_i,\ff{Y}_j)}}
\end{equation}
\end{proposition}
\begin{proof}
\begin{multline*}
\Pro{C}(\ff{X},\ff{Y}) \;=\; \cc{H}om(\cc{C},\, \cc{C}at)(\coLim{j \in \cc{J}}{\cc{C}(\ff{Y}_j,-)}, \;\coLim{i \in \cc{I}}{\cc{C}(\ff{X}_i,-))} \;\cong\;
\\
\Lim{j \in \cc{J}^{op}}{\cc{H}om(\cc{C},\, \cc{C}at)(\cc{C}(\ff{Y}_j,-), \;\coLim{i \in \cc{I}}{\cc{C}(\ff{X}_i,-))}} \;\cong\;
\Lim{j\in \cc{J}^{op}}{\coLim{i\in \cc{I}}{\cc{C}(\ff{X}_i,\ff{Y}_j)}}
\end{multline*}

The first isomorphism is due to \ref{colimits} and the second one to \ref{2Yoneda}.
\end{proof}

\begin{corollary}
The 2-category $\Pro{C}$ is locally small.
\end{corollary}
\begin{corollary} \label{CisPro}
There is a canonical 2-fully-faithful 2-functor
\mbox{$\cc{C}\mr{c} 2$-$\cc{P}ro(\cc{C})$} which sends an object of $\cc{C}$ into the corresponding \mbox{2-pro-object} with index \mbox{2-category} $\{*\}$. Since this 2-functor is also injective on objects, we can identify $\cc{C}$ with a 2-full subcategory of
$2\hbox{-}\cc{P}ro(\cc{C})$.
\cqd
\end{corollary}

Where there is no risk of confusion, we will omit to indicate notationally this identification. By the very definition of $2\hbox{-}\cc{P}ro(\cc{C})$ it follows:

\begin{proposition}\label{X=lim}
If $\ff{X}=(\ff{X}_i)_{i\in \cc{I}}$ is any 2-pro-object of $\cc{C}$, then
\mbox{$\ff{X}=\Lim{i\in \cc{I}^{op}}{\ff{X}_i}\;$} in \mbox{$2$-$\cc{P}ro(\cc{C})$.} $\ff{X}$ is equipped with projections, for each
$i \in \cc{I}$, $\ff{X} \mr{\pi_i} \ff{X}_i$, and a pseudocone structure,
for each $i \mr{u} j \, \in \cc{I}$, invertible $2$-cells
$\pi_i \Mr{\pi_u} \ff{X}_u \, \pi_j$.

 \vspace{1ex}

Under the isomorphism
$\Pro{C}(\ff{X},\, \ff{X}_i) \cong \coLim{k \in
\cc{I}}{\cc{C}(\ff{X}_k,\, \ff{X}_i})$ \eqref{basica2},
the projections $\ff{X} \mr{\pi_i} \ff{X}_i$
correspond to the object $(id_{\ff{X}_i},\, i)$ in construction \ref{defLF}.
\begin{flushright}
\cqd 
\end{flushright}

\end{proposition}
Note that from this proposition it follows:
\begin{remark} \label{XisLim}
Given any two pro-objects $\ff{X}, \, \ff{Z} \, \in $ 2-$\cc{P}ro(\cc{C})$, there is an isomorphism of categories
2-$\cc{P}ro(\cc{C})(\ff{Z},\, \ff{X}) \mr{\cong}
\ff{PC}_{2\hbox{-}\cc{P}ro(\cc{C})}(\ff{Z},\, c\ff{X})$, where
$\ff{PC}_{2\hbox{-}\cc{P}ro(\cc{C})}$ is the category of pseudocones for the 2-functor $c\ff{X}$ with \mbox{vertex
$\ff{Z}$.}
\end{remark}
\vspace{1ex}

It is important to note that when $\Lim{i\in \cc{I}^{op}}{\ff{X}_i}$ exists in $\cc{C}$, this pseudolimit would not be isomorphic to $\ff{X}$ in $2$-$\cc{P}ro(\cc{C})$. In general, the functor $c$ does not preserve 2-cofiltered pseudolimits, in fact, it will preserve them only when $\cc{C}$ is already a category of $2$-pro-objects, in which case $c$ is an equivalence.
\subsection{Lemmas to compute with 2-pro-objects.} $ $

\begin{definition} \label{representa} ${}$

\begin{enumerate}
 \item Let $\ff{X} \mr{\ff{f}} \ff{Y}$ be an arrow in $2$-$\cc{P}ro(\cc{C})$. We say that a pair $(\ff{r},\varphi)$ \emph{represents} $\ff{f}$, if $\varphi$ is an invertible $2$-cell $\pi_j \, \ff{f} \Mr{\varphi} \ff{r} \, \pi_i$. That is, if
we have the following diagram in $2$-$\cc{P}ro(\cc{C})$:
$$
\xymatrix@C=0.5pc@R=0.5pc
           {
            \ff{X} \ar[rr]^{\ff{f}}  \ar[dd]_{\pi_i}
            & & \ff{Y} \ar[dd]^{\pi_j}
           \\
            & \Downarrow \cong \varphi
            &
           \\ \ff{X}_i \ar[rr]_{\ff{r}}
            & & \ff{Y}_j
           }
$$

\vspace{-2ex}

\item Let $\ff{X} \cellrd{\ff{f}}{\alpha}{\ff{g}} \ff{Y}$ and
$\;\ff{X}_i \cellrd{\ff{r}}{\theta}{\ff{s}} \ff{Y}_j$ be 2-cells in
$2$-$\cc{P}ro(\cc{C})$ and $\cc{C}$ as in the following diagram:
$$
\xymatrix@C=9ex@R=7ex
        {
         \ff{X}   \ar@<1.6ex>[r]^{\ff{f}}
             \ar@{}@<-1.3ex>[r]^{\!\! {\alpha} \, \!\Downarrow}
             \ar@<-1.1ex>[r]_{\ff{g}}
             \ar[d]^{\pi_i}
         & \ff{Y} \ar[d]^{\pi_j}
        \\
         \ff{X}_i
             \ar@<1.6ex>[r]^{\ff{r}}
             \ar@{}@<-1.3ex>[r]^{\!\! {\theta} \, \!\Downarrow}
             \ar@<-1.1ex>[r]_{\ff{s}}
         & \ff{Y}_j
        }
$$
We say that $(\theta, \ff{r}, \varphi, \ff{s}, \psi)$ \emph{represents} $\alpha$ if $(\ff{r},\varphi)$ represents $\ff{f}$, $(\ff{s},\psi)$ represents $\ff{g}$,
and the \mbox{following} diagram commutes in $2$-$\cc{P}ro(\cc{C})$:
$$\vcenter{\xymatrix
        {
          \pi_j \ff{f} \ar@{=>}[r]^{\varphi}_\cong
                  \ar@{=>}[d]_{\pi_j\alpha}
        & \ff{r} \pi_i \ar@{=>}[d]^{\theta\pi_i}
        \\
          \pi_j \ff{g}
                  \ar@{=>}[r]_{\cong}^{\psi}
        & \ff{s} \pi_i} }
        \vcenter{\xymatrix@C=-.4pc{\quad \quad i.e. \quad \quad  }}
\vcenter{\xymatrix@C=0ex
          {
           \pi_j \dl & \dc{\varphi} & \dr \ff{f}
          \\
           \ff{r} \dcell{\theta} && \pi_i \deq
          \\
           \ff{s} && \pi_i
          }}
\vcenter{\xymatrix{\comw{\pi_j \ff{f}} \\ = \\ \comw{\pi_j \ff{f}} }}
\vcenter{\xymatrix@C=0ex
          {
           \pi_j \deq && \ff{f} \dcell{\alpha}
          \\
           \pi_j \dl & \dc{\psi} & \ff{g} \dr
          \\
           \ff{s} && \pi_i
          }}$$
          
That is, $\theta \pi_i = \pi_j \alpha$  "modulo" a pair of invertible 2-cells $\varphi,\, \psi$.

\vspace{1ex}

Clearly, if $\alpha$ is invertible, then so is $\theta$.

\end{enumerate}

\end{definition}

\begin{proposition} \label{idrepresenta} ${}$
Let $\ff{X} = (\ff{X}_i)_{i\in \cc{I}}$ and
$\ff{Y} = (\ff{Y}_j)_{j\in \cc{J}}$ be any two objects in
$2$-$\cc{P}ro(\cc{C})$:
\begin{enumerate}

\item Let $\ff{X} \mr{\ff{f}} \ff{Y}$, then, for any $j \in \cc{J}$ there is an
$i \in \cc{I}$ and $\ff{X}_i \mr{\ff{r}} \ff{Y}_j$ in $\cc{C}$, such that $(\ff{r},id)$ represents $\ff{f}$.

\item Let $\ff{X} \cellrd{\ff{f}}{\alpha}{\ff{g}} \ff{Y}$, then, for any $j \in \cc{J}$ there is an
$i \in \cc{I}$, $\ff{X}_i \cellrd{\ff{r}}{\theta}{\ff{s}} \ff{Y}_j$ in
$\cc{C}$, and \mbox{appropriate} invertible 2-cells $\varphi$ and $\psi$ such that $(\theta,\ff{r},\varphi,\ff{s},\psi)$ represents $\alpha$.

\end{enumerate}

\end{proposition}
\begin{proof}
Consider $\ff{X}\cellrd{\pi_j \ff{f}}{\pi_j\alpha}{\pi_j \ff{g}}\ff{Y}_j$ and use formula \ref{iso} plus the constructions of pseudolimits and  2-filtered pseudocolimits, \ref{limincat}, \ref{defLF}.
\end{proof}
\begin{lemma}\label{lema1}
Let
$\ff{X}=(\ff{X}_i)_{i\in \cc{I}}\in 2\hbox{-}\cc{P}ro(\mathcal{C})$, let $\ff{X}_i\mr{\ff{r}} \ff{C}$,
\mbox{$\ff{X}_{j}\mr{\ff{s}} \ff{C}\in \cc{C}$}, and let
$\ff{X} \cellrd{\ff{r} \pi_i}{\alpha}{\ff{s} \pi_{j}} \ff{C}\in 2$-$\cc{P}ro(\mathcal{C})$. Then, $\exists \vcenter{\xymatrix@R=-0.6pc@C=0.8pc{i \ar[rd]^{u} &\\& k \\ j\ar[ru]_{v}}}$ and $\ff{X}_{k} \cellrd{\ff{r} \ff{X}_u}{\theta}{\ff{s} \ff{X}_v} \ff{C}$ such that the following diagram commutes in $\Pro{C}$:

$$\vcenter
    {
     \xymatrix@R=3ex
     {
     \ff{X}\ar[rr]^{\pi_i} \ar[d]_{\pi_{k}}
     &&
     \ff{X}_i \ar[dd]^{\ff{r}}
     \\
     \ff{X}_{k} \ar[rru]_{\ff{X}_u} \ar@{}[ru]|>>>>>>{\displaystyle \pi_u \!\Downarrow} \ar[d]_{\ff{X}_v}
     &
     \ar@{}[d]|<<<{\displaystyle \Downarrow \theta}
     \\
     \ff{X}_{j} \ar[rr]_{\ff{s}}
     &&
     \ff{C}
      }
     }
 \vcenter{\xymatrix@C=-.4pc{\quad = \quad}}
 \vcenter
   {
    \xymatrix@R=3ex
     {
     & \ff{X} \ar[dl]_{\pi_{k}}
              \ar[rr]^{\pi_i}
              \ar[dd]^{\pi_{j}}
     &&
     \ff{X}_i \ar[dd]^{\ff{r}}
     \\
     \ff{X}_{k} \ar[dr]_{\ff{X}_v} \ar@<-1ex>@{}[r]|>>>>>>{\displaystyle \Leftarrow}  \ar@<1ex>@{}[r]|>>>>>>{\displaystyle \pi_v} &
     &
     \Downarrow  \alpha
     &
     \\ &
     \ff{X}_{j} \ar[rr]_{\ff{s}}
     &&
     \ff{C}
     }
   }
 $$
$$
\vcenter
  {
   \xymatrix
      {
        \ff{r}\,\pi_i  \ar@{=>}[r]^{\ff{r} \pi_u}_\cong
                       \ar@{=>}[d]_{\alpha}
      & \ff{r}\,\ff{X}_u \pi_{k}
                       \ar@{=>}[d]^{\theta\pi_{k}}
      \\
        \ff{s} \pi_{j} \ar@{=>}[r]^{\ff{s}\pi_v}_\cong
      & \ff{s} \ff{X}_v  \pi_{k}
      }
  }
\vcenter{\xymatrix@C=-.4pc{\quad \quad i.e. \quad \quad  }}
\vcenter{\xymatrix@C=-2ex
          {
           \ff{r} \deq & & & \pi_i \dcellop{\pi_u} & 
          \\
           \ff{r} \dl & \comw{\ff{X}_u} \dc{\theta} & \ff{X}_u \dr & & \pi_k \deq
          \\
           \ \ff{s} \  && \ff{X}_v & & \pi_k
          }}
\vcenter{\xymatrix{\comw{\pi_j \ff{f}} \\ = \\ \comw{\pi_j \ff{f}} }}
\vcenter{\xymatrix@C=-2ex
          {
           \ff{r} \dl && \ar@{}[dl]|{\alpha} & \pi_i \dr 
          \\
           \ff{s} \deq &  &  & \pi_j \dcellop{\pi_v}
          \\
           \ff{s} & \comw{\ff{X}_v} & \ff{X}_v & & \pi_k
          }}$$
Clearly, if $\alpha$ is invertible, then so is $\theta$.
\end{lemma}

\begin{proof}
By formula $\ref{iso}$ and the construction of 2-filtered pseudocolimits (\ref{defLF}), $\alpha$ corresponds to a $(\ff{r},i)\mr{[u,\theta,v]} (\ff{s},j)\in \coLim{i\in \cc{I}}{\cc{C}(\ff{X}_i,\ff{C})}$ . So, $\exists \vcenter{\xymatrix@R=-0.6pc@C=0.8pc{i \ar[rd]^{u} &\\& k \\ j\ar[ru]_{v}}}$ and $\ff{X}_{k} \cellrd{\ff{r} \ff{X}_u}{\theta}{\ff{s} \ff{X}_v} \ff{C}$ such that $\pi_v \ff{s}\circ \alpha=\pi_k \theta \circ \pi_u \ff{r}$\ \ \ i.e.

\hspace{8ex} 
$\vcenter{\xymatrix@C=-2ex
          {
           \ff{r} \deq & & & \pi_i \dcellop{\pi_u} & 
          \\
           \ff{r} \dl & \comw{\ff{X}_u} \dc{\theta} & \ff{X}_u \dr & & \pi_k \deq
          \\
           \ \ff{s} \  && \ff{X}_v & & \pi_k
          }}
\vcenter{\xymatrix{\comw{\pi_j \ff{f}} \\ = \\ \comw{\pi_j \ff{f}} }}
\vcenter{\xymatrix@C=-2ex
          {
           \ff{r} \dl && \ar@{}[dl]|{\alpha} & \pi_i \dr 
          \\
           \ff{s} \deq &  &  & \pi_j \dcellop{\pi_v}
          \\
           \ff{s} & \comw{\ff{X}_v} & \ff{X}_v & & \pi_k
          }}$, \;\; as we wanted to prove.
\end{proof}

The following is an immediate consequence of \cite[Lemma 2.2.]{DS}

\begin{remark} \label{u=v}
If $i=j$, then one can choose $u=v$. \cqd
\end{remark}
\begin{lemma}\label{lema4}

Let $\ff{X}=(\ff{X}_i)_{i\in \cc{I}}\in 2$-$\cc{P}ro(\cc{C})$ and $\ff{X}_i\cellpairrd{\ff{f}}{\theta}{\theta'}{\ff{g}}\ff{C} \in \cc{C}$ be such that $\theta\pi_i=\theta'\pi_i$ in $2$-$\cc{P}ro(\cc{C})$. Then $\exists \ i\mr{u} i'$ such that $\theta \ff{X}_u=\theta' \ff{X}_u$.
\end{lemma}

\begin{proof}
It follows from \ref{basica2} and \cite[Lemma 1.20]{DS}.
\end{proof}

\begin{lemma}\label{lema5}
Let $\ff{X} \cellrd{\ff{f}}{\alpha}{\ff{g}} \ff{Y}$ in $2$-$\cc{P}ro(\cc{C})$ and
$\ff{X}_i \cellpairrd{\ff{r}}{\theta}{\theta'}{\ff{s}} \ff{Y}_j$ in $\cc{C}$ such that $(\theta,\,\ff{r},\,\varphi,\,\ff{s},\,\psi)$ and $(\theta',\,\ff{r},\,\varphi,\,\ff{s},\,\psi)$ both represents $\alpha$. Then, there exists $i \mr{u} i' \in \cc{I}$ such that $\theta \ff{X}_u=\theta' \ff{X}_u$.
\end{lemma}

\begin{proof}
Since both $(\theta,\ff{r},\varphi,\ff{s},\psi)$ and $(\theta',\ff{r},\varphi,\ff{s},\psi)$ represents $\alpha$, and $\varphi\hbox{, }\psi$ are invertible, it follows that  $\theta\pi_i=\theta'\pi_i$. Then, by \ref{lema4}, there exists $i \mr{u} i' \in \cc{I}$ such that $\theta \ff{X}_u=\theta' \ff{X}_u$.
\end{proof}

\begin{lemma}\label{lema2}
Let $\ff{X} \cellrd{\ff{f}}{\alpha}{\ff{g}} \ff{Y} \in $ $2$-$\cc{P}ro(\cc{C})$, $(\ff{r},\varphi)$ representing $\ff{f}$, \mbox{$\ff{X}_i\mr{\ff{r}} \ff{Y}_j$} and $(\ff{s},\psi)$ representing $\ff{g}$, $\ff{X}_{i'}\mr{\ff{s}} \ff{Y}_j$. Then, $\exists \vcenter{\xymatrix@R=-0.6pc@C=0.8pc{i \ar[rd]^{u} &\\& k \\ i'\ar[ru]_{v}}}$ and $\ff{X}_{k} \cellrd{\ff{r}\ff{X}_u}{\theta}{\ff{s}\ff{X}_v} \ff{Y}_j$  such that $(\theta,\; \ff{r} \ff{X}_u,\; \ff{r} \pi_u \circ \varphi, \; \ff{s}\ff{X}_v, \; \ff{s} \pi_v \circ \psi)$ represents $\alpha$.
Clearly, if $\alpha$ is invertible, then so is $\theta$.
\end{lemma}
\begin{proof}
In lemma $\ref{lema1}$, take $\ff{C}=\ff{Y}_j$, and
$\alpha = \psi \circ \pi_j\alpha \circ \varphi^{-1}$. Then, $\exists \vcenter{\xymatrix@R=-0.6pc@C=0.8pc{i \ar[rd]^{u} &\\& k \\ i'\ar[ru]_{v}}}$ and $\ff{X}_{k} \cellrd{\ff{r}\ff{X}_u}{\theta}{\ff{s}\ff{X}_v} \ff{Y}_j$ such that $
 \; \theta \pi_k \circ \ff{r} \pi_u \,=\, s\pi_v \circ \psi \circ \pi_j \alpha \circ \varphi^{-1} \,,$ or equivalently
$
 \; \theta \pi_k \circ \ff{r} \pi_u \circ \varphi \,=\, s\pi_v \circ \psi \circ \pi_j \alpha \,,$ i.e. the following diagram commutes in $\Pro{C}$:

\vspace{-2ex}

$$
\xymatrix@R=1ex
       {
        & {\ff{r} \pi_i} \ar@{=>}[r]^{\ff{r} \pi_u}
        & {\ff{r} \ff{X}_u \pi_k}
                         \ar@{=>}[rd]^{\theta \pi_k}
        \\
          {\pi_j \ff{f}} \ar@{=>}[ru]^{\varphi}
                         \ar@{=>}[rd]^{\pi_j \alpha}
        &&& {\ff{s} \ff{X}_v \pi_k}
        \\
        & {\pi_j \ff{g}} \ar@{=>}[r]^{\psi}
        & {\ff{s} \pi_{i'}} \ar@{=>}[ru]^{\ff{s} \pi_v}
       }
       \vcenter{\xymatrix@C=-.4pc{\quad i.e.  \quad  }}
\vcenter{\xymatrix@C=-2ex
          {\pi_j \dl && \ar@{}[dl]|{\varphi} & \ff{f} \dr \\
          \ff{r} \deq &&& \pi_i \dcellop{\pi_u} \\
          \ff{r} \dl & \comw{\ff{X}_u} \dc{\theta} & \ff{X}_u \dr && \pi_k \deq \\
          \ff{s} && \ff{X}_v && \pi_k           }
          }
\vcenter{\xymatrix{\comw{\pi_j \ff{f}} \\ = \\ \comw{\pi_j \ff{f}} }}
\vcenter{\xymatrix@C=-2ex
          {\pi_j \deq &&& \ff{f} \dcell{\alpha} \\
          \pi_j \dl && \ar@{}[dl]|{\psi} & \ff{g} \dr \\
          \ff{s} \deq &&& \pi_{i'} \dcellop{\pi_v} \\
          \ff{s} &\comw{\ff{X}_v} & \ff{X}_v && \pi_k       
          }}
$$

This proves that $(\theta,\; \ff{r} \ff{X}_u,\; \ff{r} \pi_u \circ \varphi, \; \ff{s}\ff{X}_v, \; \ff{s} \pi_v \circ \psi)$ \mbox{represents $\alpha$.}
\end{proof}
From remark \ref{u=v} we have:
\begin{remark}
If $i=i'$, then one can choose $u=v$. \cqd
\end{remark}

\vspace{0.1ex}
 
\subsection{2-cofiltered pseudolimits in $2\hbox{-}\cc{P}ro(\cc{C})$.}  $ $

Let $\cc{J}$ be a small  2-filtered 2-category and
$\cc{J}^{op} \mr{\ff{X}}  2\hbox{-}\cc{P}ro(\cc{C})$ a \mbox{2-functor,}
\mbox{$\ff{X}^j = (\ff{X}^j_i)_{i\in \cc{I}_j}$,} $\cc{I}_j^{op} \mr{\ff{X}^j} \cc{C}$. Recall (\ref{X=lim}) that for each $j$ in $\cc{J}$, $\ff{X}^j$ is equipped with a pseudolimit pseudocone $\{\pi^j_i\}_{i \in \cc{I}_j}$,
$\{\pi^j_u\}_{i \mr{u} i' \in \cc{I}_j}$ for the \mbox{2-functor}
$\ff{X}^j$.

We are going to construct a 2-pro-object which is going to be the pseudolimit of
$\ff{X}$ in $2\hbox{-} \cc{P}ro(\cc{C})$. First we construct its index category
\begin{definition}\label{kequis}
Let $\cc{K}_\ff{X}$ be the 2-category consisting on:
\vspace{2ex}
\begin{enumerate}
	\item 0-cells of $\cc{K}_\ff{X}$: $(i,j)$, where $j\in \cc{J}$, $i\in \cc{I}_j$.
	\vspace{1ex}
	\item 1-cells of $\cc{K}_\ff{X}$: $(i,j)\mr{(a,\ff{r},\varphi)} (i',j')$, where $j\smr{a} j'\in \cc{J}$, $\ff{X}_{i'}^{j'}\smr{\ff{r}} \ff{X}_{i}^{j}$ are such that $(\ff{r},\varphi)$ represents $\ff{X}^a$.
	\vspace{1ex}
	\item 2-cells of $\cc{K}_\ff{X}$: $(a,\ff{r},\varphi)\Mr{(\alpha,\theta)}(b,\ff{s},\psi)$, where $a \Mr{\alpha} b \in \cc{J}$ and $(\theta,\ff{r},\varphi,\ff{s},\psi)$ represents $\ff{X}^\alpha$.
\end{enumerate}
\vspace{2ex}

The 2-category structure is given as follows:

$$
\xymatrix@C=6pc
             {
              (i,j) \ar@<4.8ex>[r]^{(a,\ff{r},\varphi)}
                    \ar@{}@<3.5ex>[r]|{\Downarrow(\alpha,\theta)}
                    \ar[r]^{(b,\ff{s},\psi)}
                    \ar@{}@<-1.3ex>[r]|{\Downarrow(\beta,\eta)}
                    \ar@<-4.8ex>[r]^{(c,\ff{t},\phi)}
                 &
             (i',j') \ar@<4.8ex>[r]^{(a',\ff{r}',\varphi')}
                     \ar@{}@<3.5ex>[r]|{\Downarrow(\alpha',\theta')}
                     \ar[r]^{(b', \ff{s}',\psi')}
                     \ar@{}@<-1.3ex>[r]|{\Downarrow(\beta',\eta')}
                     \ar@<-4.8ex>[r]^{(c', \ff{t}',\phi')}
                 &
             (i'',j'')
             }
$$

\vspace{2ex}
\begin{enumerate}
\item $(a',\ff{r}',\varphi')(a,\ff{r},\varphi)=(a' a,\ff{r}\ff{r}',\ff{r}\varphi'\circ \varphi \ff{X}^{a'})$
\vspace{1ex}
\item $(\alpha',\theta')(\alpha,\theta)=(\alpha'\alpha,\theta\theta')$
\vspace{1ex}
\item $(\beta,\eta)\circ (\alpha,\theta)=(\beta\circ \alpha,\eta\circ\theta)$

\end{enumerate}

One can easily check that the structure so defined is indeed a 2-category, which  is clearly small.\end{definition}

\begin{proposition}
The 2-category $\cc{K}_\ff{X}$ is 2-filtered.
\end{proposition}

\vspace{-2ex}

\begin{proof}

\vspace{-1ex}

$F0.$ Let $(i,j)$,$(i',j')\in \cc{K}_\ff{X}$. Since $\cc{J}$ is 2-filtered, $\exists \vcenter{\xymatrix@R=-0.6pc@C=0.8pc{j \ar[rd]^{a} &\\& j'' \\ j'\ar[ru]_{b}}}$. By \ref{idrepresenta}, $\exists\  \ff{X}_{i_1}^{j''}\mr{\ff{r}_1}{} \ff{X}_i^j$ and $ \ff{X}_{i_2}^{j''}\mr{\ff{r}_2}{} \ff{X}_{i'}^{j'}$ such that $(\ff{r}_1,id)$ represents $\ff{X}^{a}$ and $(\ff{r}_2,id)$ represents $\ff{X}^{b}$.
Since $\cc{I}_{j''}$ is 2-filtered, $\exists \vcenter{\xymatrix@R=-0.6pc@C=0.8pc{i_1 \ar[rd]^{u} &\\ & i'' \\ i_2\ar[ru]_{v}}}$. Then, we have the following situation in $\cc{K}_\ff{X}$ which proves $F0.$:
$$
\xymatrix@R=.3pc@C=6pc
         {
           (i,j) \hspace{2ex} \ar[rd]^*[l]
                 {
                  \hspace{-5ex} (a,\ff{r}_1 \ff{X}^{j''}_{u},\ff{r}
                                                        _1{\pi_{u}^{j''}})
                 }
           &
           \\
           & \hspace{3ex} (i'',j'')
           \\ (i',j') \hspace{2ex} \ar[ru]_*[l]
                 {
                  \hspace{-5ex} (b,\ff{r}_2 \ff{X}^{j''}_{v},\ff{r}
                                                        _2{\pi_{v}^{j''}})
                 }
          }
$$

\vspace{-1ex}

$F1.$ Let $\xymatrix{(i,j)\ar@<1ex>[r]^{(a,\ff{r},\varphi)} \ar@<-1ex>[r]_{(b,\ff{s},\psi)} & (i',j')}\in \cc{K}_\ff{X}$. Since $\cc{J}$ is 2-filtered, $\exists \  j'\mr{c} j''$ and an invertible 2-cell $ca \Mr{\alpha}cb$. By \ref{idrepresenta}, $\exists \ \ff{X}_{k}^{j''}\mr{\ff{t}} \ff{X}_{i'}^{j'}$ such that $(\ff{t},id)$ \mbox{represents} $\ff{X}^c$. Then $(\ff{r}\ff{t},\varphi \ff{X}^c)$ represents $\ff{X}^{ca}$ and $(\ff{s}\ff{t},\psi \ff{X}^c)$ \mbox{represents} $\ff{X}^{cb}$, so, by \ref{lema2},
there \mbox{exists} \mbox{$k \mr{w} i''\in \cc{I}_{j''}$} and an invertible 2-cell \mbox{$\ff{r}\ff{t}\ff{X}^{j''}_{w}\Mr{\theta}\ff{s}\ff{t}\ff{X}^{j''}_{w}$} such that
\mbox{$(\theta,\,\ff{r}\ff{t}\ff{X}^{j''}_{w},\,\ff{r}\ff{t}\pi_{w} \circ \varphi \ff{X}^c,\,\ff{s}\ff{t}\ff{X}^{j''}_{w},\,\ff{s}\ff{t}\pi_{w} \circ \psi \ff{X}^c)$}
represents $\ff{X}^{\alpha}$.
Then we have an invertible 2-cell in $\cc{K}_\ff{X}$ $\xymatrix{(i,j)\ar@<1.5ex>[rr]^{(c,\ff{t}\ff{X}^{j''}_{w},t\pi_w)(a,\ff{r},\varphi)} \ar@<-1.5ex>[rr]_{(c,\ff{t}\ff{X}^{j''}_{w},\ff{t}\pi_{w})(b,\ff{s},\psi)} & \Downarrow (\alpha,\theta)& (i'',j'')}$ \mbox{which proves $F1.$}

$F2.$ Let $\xymatrix@C=0.2pc{(i,j)\ar@<1.5ex>[rr]^{(a,\ff{r},\varphi)} \ar@<-1.5ex>[rr]_{(b,\ff{s}, \psi)} & \Downarrow (\alpha,\theta) \Downarrow (\alpha',\theta')& (i',j')} \in \cc{K}_\ff{X}$. Since $\cc{J}$ is \mbox{2-filtered,} $\exists\ j' \mr{c} j'' \in \cc{J}$ such that \mbox{$c\alpha=c\alpha'$}. Also, by \ref{idrepresenta}, $\exists\  \ff{X}_{k}^{j''}\mr{\ff{t}} \ff{X}_{i'}^{j'}$ such that $(\ff{t},id)$ represents $\ff{X}^c$. Then, it is easy to check that $(\ff{t},\ff{t},id,\ff{t},id)$ represents $\ff{X}^c$ and therefore we have that $(\theta \ff{t},\ff{r}\ff{t},\varphi \ff{X}^c,\ff{s}\ff{t},\psi \ff{X}^c)$ and $(\theta'\ff{t},\ff{r}\ff{t},\varphi \ff{X}^c,\ff{s}\ff{t},\psi \ff{X}^c)$ both represent $\ff{X}^{c\alpha}$: 

$$ \vcenter{\xymatrix@C=-0.3pc {\pi_i \dl & \dc{\varphi} & \ff{X}^a \dr && \ff{X}^c \deq \\
			      \ff{r} \deq && \pi_{i'} \dl & \dc{=} & \ff{X}^c \dr \\
			      \ff{r} \dcell{\theta} && \ff{t} \deq & \! \comw{r} \! & \pi_k \deq \\
			      \ff{s} && \ff{t} && \pi_k }}
 \vcenter{\xymatrix@C=1pc {\ \ \ = \ \ \ }}
  \vcenter{\xymatrix@C=-.3pc {\pi_i \dl & \dc{\varphi} & \ff{X}^a \dr && \ff{X}^c \deq \\
			      \ff{r} \dcell{\theta} && \pi_{i'} \deq &  & \ff{X}^c \deq \\   
			      \ff{s}  \deq && \pi_{i'} \dl & \dc{=} & \ff{X}^c \dr\\
			      \ff{s} && \ff{t} & \! \comw{r} \! & \pi_k  }}
   \vcenter{\xymatrix@C=1pc {\ \ \ = \ \ \ }}
    \vcenter{\xymatrix@C=-.3pc { \pi_i \deq && \ff{X}^a \dcellb{\ff{X}^{\alpha} } && \ff{X}^c \deq \\
				  \pi_i \dl & \dc{\psi} & \ff{X}^b \dr && \ff{X}^c \deq \\
				  \ff{s} \deq && \pi_{i'} \dl &\dc{=} & \ff{X}^c \dr \\
				  \ff{s} && \ff{t} & \! \comw{r} \! & \pi_k }}$$

\noindent where the second equality is due to the fact that $(\theta,\ff{r},\varphi,\ff{s},\psi)$ represents $\ff{X}^\alpha$.

Then, by \ref{lema5}, $\exists \, k \mr{w} i'' \in \cc{I}_{j''}$ such that \mbox{$\theta  \ff{t} \ff{X}^{j''}_{w}=\theta'\ff{t}\ff{X}^{j''}_{w}$}, so $(c,\ff{t}\ff{X}^{j''}_{w},\ff{t}\pi_{w})(\alpha,\theta)=(c,\ff{t}\ff{X}^{j''}_{w},t\pi_{w})(\alpha',\
theta')$, which proves $F2$.
\end{proof}

\begin{theorem}\label{teo} Let $\widetilde{\ff{X}}$ be the 2-pro-object
$\cc{K}_\ff{X}^{op} \mr{\widetilde{\ff{X}}} \cc{C}$ defined by
\mbox{$\widetilde{\ff{X}}_{(i, j)} = \ff{X}^j_i$,}
$\widetilde{\ff{X}}_{(a,\ff{r}, \varphi)} = \ff{r}$, and
$\widetilde{\ff{X}}_{(\alpha, \theta)} = \theta$.
Then the following equation holds in $2$-$\cc{P}ro(\cc{C})$:
$$\widetilde{\ff{X}}=\Lim{j\in \cc{J}^{op}}{\ff{X}^j}$$
\end{theorem}


\begin{proof}
Let $\ff{Z} \in$ 2-$\cc{P}ro(\cc{C})$, and
$\{\ff{Z}\mr{\ff{h}_j} \ff{X}^j\}_{j\in \cc{J}}$,
\mbox{$\{\ff{h}_j \Mr{\ff{h}_a} \ff{X}^a \ff{h}_{j'}\}_{j \mr{a}j' \in \cc{J}}$}
be a pseudocone for $\ff{X}$ with vertex $\ff{Z}$ (\ref{pseudocone}). Given
\mbox{$(i,j)\mr{(a,\ff{r},\varphi)} (i',j') \; \in \cc{K}_\ff{X}$}, the definitions
$\ff{h}_{(i,j)} = \pi^j_i \ff{h}_j$ and
$\ff{h}_{(a,\ff{r},\varphi)} = \varphi \ff{h}_{j'} \circ \pi^j_i \ff{h}_a$ determine a pseudocone for $c\widetilde{\ff{X}}$ with vertex $\ff{Z}$:

\noindent PC0: It's clear.

\noindent PC1: Given $(i,j) \mr{(a,\ff{r},\varphi)} (i',j') \mr{(b, \ff{s}, \psi)} (i'',j'')$,

$$\vcenter{ \xymatrix@C=-0.3pc{ \pi_i^j \deq &&&& \ff{h}_j \ar@{-}[dll] \ar@{-}[drr] \dc{\ff{h}_{ba}} \\
			       \pi_i^j \dl & \dc{\varphi} & \ff{X}^a \dr && \ff{X}^b \deq && \ff{h}_{j''} \deq \\
			       \ff{r} \deq && \pi_{i'}^{j'} \dl & \dc{\psi} & \ff{X}^b \dr && \ff{h}_{j''} \deq \\
			       \ff{r} && \ff{s} && \pi_{i''}^{j''} && \ff{h}_{j''} } }
\vcenter{ \xymatrix@C=-0pc{ \  = \  } }
\vcenter{ \xymatrix@C=-0.3pc{ \pi_i^j \deq &&&& \ff{h}_j \ar@{-}[dll] \ar@{-}[dr] \ar@{}[d]|{\ff{h}_a} \\
			     \pi_i^j \deq && \ff{X}^a \deq &&& \ff{h}_{j'} \op{\ff{h}_b} \\
			     \pi_i^j \dl & \dc{\varphi} & \ff{X}^a \dr && \ff{X}^b \deq && \ff{h}_{j''} \deq \\
			     \ff{r} \deq && \pi_{i'}^{j'} \dl & \dc{\psi} & \ff{X}^b \dr && \ff{h}_{j''} \deq \\
			     \ff{r} && \ff{s} && \pi_{i''}^{j''} && \ff{h}_{j''}   } }
\vcenter{ \xymatrix@C=-0pc{ \  = \  } }
\vcenter{ \xymatrix@C=-0.3pc{   \pi_i^j \deq &&&& \ff{h}_j \ar@{-}[dll] \ar@{-}[dr] \ar@{}[d]|{\ff{h}_a} \\
			     \pi_i^j \dl & \dc{\varphi} & \ff{X}^a \dr &&& \ff{h}_{j'} \deq \\
			     \ff{r} \deq && \pi_{i'}^{j'} \deq &&& \ff{h}_{j'} \op{\ff{h}_b} \\
			     \ff{r} \deq && \pi_{i'}^{j'} \dl & \dc{\psi} & \ff{X}^b \dr && \ff{h}_{j''} \deq \\
			     \ff{r} && \ff{s} && \pi_{i''}^{j''} && \ff{h}_{j''}   } }$$

\noindent where the first equality is due to the fact that $\ff{h}$ is a pseudocone.	

\noindent PC2: Given $(i,j) \cellrd{(a,\ff{r},\varphi)}{(\alpha,\theta)}{(b,\ff{s},\psi)} (i',j')$,

$$\vcenter{ \xymatrix@C=-0.3pc{ \pi_i^j \deq &&& \ff{h}_j \op{\ff{h}_a} \\
				\pi_i^j \dl & \dc{\varphi} & \ff{X}^a \dr && \ff{h}_{j'} \deq \\
				\ff{r} \dcell{\theta} && \pi_{i'}^{j'} \deq && \ff{h}_{j'} \deq \\
				\ff{s} && \pi_{i'}^{j'} && \ff{h}_{j'}    }}
\vcenter{ \xymatrix@C=0pc{ \  = \ }}
\vcenter{ \xymatrix@C=-0.3pc{  \pi_i^j \deq &&& \ff{h}_j \op{\ff{h}_a} \\
				\pi_i^j \deq && \ff{X}^a \dcell{\ff{X}^{\alpha}} && \ff{h}_{j'} \deq \\
				\pi_i^j \dl & \dc{\psi} & \ff{X}^b \dr && \ff{h}_{j'} \deq \\
				\ff{s} && \pi_{i'}^{j'} && \ff{h}_{j'}    }}
\vcenter{ \xymatrix@C=0pc{\  = \ }}
\vcenter{ \xymatrix@C=-0.3pc{  \pi_i^j \deq &&& \ff{h}_j \op{\ff{h}_b} \\
				\pi_i^j \dl & \dc{\psi} & \ff{X}^b \dr && \ff{h}_{j'} \deq \\
				\ff{s} && \pi_{i'}^{j'} && \ff{h}_{j'}    }} $$

\noindent where the first equality is due to the fact that $(\theta,\ff{r},\varphi, \ff{s},\psi)$ represents $\ff{X}^\alpha$ and the second one is valid because $\ff{h}$ is a pseudocone.
				
It is straightforward to check that this extends to a functor, that we denote $p$ (for the isomorphism below see \ref{XisLim}):
$$
\ff{PC}_{2\hbox{-}\cc{P}ro(\cc{C})}(\ff{Z},\ff{X}) \mr{p} \ff{PC}_{2\hbox{-}\cc{P}ro(\cc{C})}(\ff{Z},c\widetilde{\ff{X}})
 \, \cong \,2\hbox{-} \cc{P}ro(\cc{C})(\ff{Z},\, \ff{\widetilde{X}})
$$

The theorem follows if $p$ is an isomorphism. In the sequel we prove that, in fact, $p$ is an isomorphism.
Let $\ff{Z} \in$ 2-$\cc{P}ro(\cc{C})$, and
$$
    {
     \{\ff{h}_{(i,j)} \Mr{\ff{h}_{(a,\ff{r},\varphi)}}
     \ff{\widetilde{X}}_{(a,\ff{r},\varphi)} \ff{h}_{(i',j')}
     = \ff{r} \, \ff{h}_{(i,j)}\}_
     {(i,j) \mr{(a,\ff{r},\varphi)} (i',j') \in \cc{K}_\ff{X}}, \hspace{2ex}
    }
\{\ff{Z}\mr{\ff{h}_{(i,j)}} \ff{X}^j_i\}_{(i,j) \in \cc{K}_\ff{X}}
$$
be a pseudocone for $c\ff{\widetilde{X}}$ with vertex $\ff{Z}$ (\ref{pseudocone}).

\vspace{1ex}

\emph{1. $p$ is bijective on objects}:

Check that for each $j \in \cc{J}$,
$\{\ff{Z}\mr{\ff{h}_{(i,j)}}\ff{X}_i^j \}_{i\in \cc{I}_j}$
together  with
\mbox{$\{\ff{h}_u = \ff{h}_{(j,\ff{X}_{u}^j,\pi_u^j)}: \ff{h}_{(i,j)} \Mr{} \ff{X}_u^j \ff{h}_{(i',j')}\}_{i\mr{u} i'\in \cc{I}_j}$} is a pseudocone for $\ff{X}^j$.
Then, since $\ff{X}^j \mr{\pi^j_i} \ff{X}^j_i$ is a pseudolimit pseudocone, it follows that there exists a unique $\ff{Z} \mr{\ff{h}_j} \ff{X}^j$ such that
\begin{equation}\label{*1}
 \forall i\in \cc{I}_j \;\; \pi_i^j\ff{h}_j=\ff{h}_{(i,j)} \;\;and\;\; \forall \; i \mr{u} i'\in \cc{I}_j \ \;\; \pi_u^j\ff{h}_j=\ff{h}_{u}.
\end{equation}
It only remains to define the 2-cells of the pseudocone structure. That is, for each
$j \mr{a} j' \in \cc{J}$, we need invertible 2-cells
$\ff{h}_j \Mr{\ff{h}_a} \ff{h}_{j'}\ff{X}^a$,
 such that $\{\ff{h}_j\}_{j\in \cc{J}}$ together with
 $\{\ff{h}_a\}_{j \mr{a} j' \in \cc{J}}$ form a pseudocone for $\ff{X}$ with vertex $\ff{Z}$.

Consider the pseudocone
 $\{\ff{X}^j \mr{\pi^j_i} \ff{X}^j_i \}_{i \in \cc{I}_j}$. Then the composites $\pi^j_i \ff{h}_j$, $\pi^j_i \ff{X}^a \ff{h}_{j'}$, determine two pseudocones
 $\{\ff{Z} \mrpair{\pi_i^j\ff{h}_j}{\pi_i^j\ff{X}^a \ff{h}_{j'}}
                                           \ff{X}_i^j\}_{i\in \cc{I}_j}$
for $\ff{X}^j$ with \mbox{vertex $\ff{Z}$.}

{\bf Claim 1}
 \emph
   { Let $(\ff{r},\varphi)$ and $(\ff{s},\psi)$ be two pairs
     representing $\ff{X}^a$ as follows:
    $$
     \xymatrix@C=0.5pc@R=0.5pc
           {
            \ff{X}^{j'} \ar[rr]^{\ff{X}^a}  \ar[dd]_{\pi^{j'}_{i'}}
            & & \ff{X}^j \ar[dd]^{\pi^j_i}
           \\
            & \Downarrow \cong \varphi
            &
           \\ \ff{X}^{j'}_{i'} \ar[rr]_{\ff{r}}
            & & \ff{X}^j_i
           }
     \hspace{5ex}
     \xymatrix@C=0.5pc@R=0.5pc
           {
            \ff{X}^{j'} \ar[rr]^{\ff{X}^a}  \ar[dd]_{\pi^{j'}_{i''}}
            & & \ff{X}^j \ar[dd]^{\pi^j_i}
           \\
            & \Downarrow \cong \psi
            &
           \\ \ff{X}^{j'}_{i''} \ar[rr]_{\ff{s}}
            & & \ff{X}^j_i
           }
$$
Then,
    \mbox{$\varphi^{-1}\ff{h}_{j'} \circ \ff{h}_{(a,\ff{r},\varphi)} =
    \psi^{-1}\ff{h}_{j'} \circ \ff{h}_{(a,\ff{s},\psi)}$}
$$\vcenter{\xymatrix@C=-0.5pc{  & \pi_i^j \ar@{-}[dl] && \ar@{}[dl]|{\ff{h}_{(a,\ff{r},\varphi)}} & \ff{h}_j \ar@{-}[dr] \\                               
			      \ff{r} \dl & \dc{\varphi^{-1}} & \pi_{i'}^{j'} \dr &&& \ff{h}_{j'}  \deq \\
			      \pi_i^j & & \ff{X}^a &&& \ff{h}_{j'} } }
\vcenter{\xymatrix@C=0pc{\ = \ \ } }
\vcenter{\xymatrix@C=-0.5pc{  & \pi_i^j \ar@{-}[dl] && \ar@{}[dl]|{\ff{h}_{(a,\ff{s},\psi)}} & \ff{h}_j \ar@{-}[dr] \\                               
			      \ff{s} \dl & \dc{\psi^{-1}} & \pi_{i''}^{j'} \dr &&& \ff{h}_{j'}  \deq \\
			      \pi_i^j & & \ff{X}^a &&& \ff{h}_{j'} } }
$$    
    }
\hfill (proof below).
\vspace{1ex}

{\bf Claim 2}
 \emph
      {
       For each $i \in \cc{I}_j$, let $(\ff{r}, \varphi)$ be a pair representing
       $\ff{X}^a$, and set
       \mbox{$\rho_i=\varphi^{-1}\ff{h}_{j'} \circ \ff{h}_{(a,\ff{r},
                                                             \varphi)}$.}
       Then, $\{\rho_i\}_{i \in \cc{I}_j}$ determines an isomorphism of
       pseudocones
       $\{\ff{Z} \cellrd{\pi_i^j\ff{h}_j}{\rho_i}
        {\pi_i^j\ff{X}^a \ff{h}_{j'}} \ff{X}^j_i\}_{i\in \cc{I}_j}$
      }
\hfill\mbox{(proof below)}.

Since $\ff{X}^j \mr{\pi^j_i} \ff{X}^j_i$ is a pseudolimit pseudocone, the functor
\mbox{$2$-$\cc{P}ro(\cc{C})(\ff{Z}, \ff{X}^j) \mr{(\pi^j)_*} \ff{PC}_{2\hbox{-}\cc{P}ro(\cc{C})}(\ff{Z}, \ff{X}^j)$}
is an isomorphism of categories. Then, from Claim 2 it follows that there are invertible 2-cells
\mbox{$\ff{Z} \cellrd{\ff{h}_j}{h_a}{\ff{X}^a\ff{h}_{j'}} \ff{X}^j \in 2$-$\cc{P}ro(\cc{C})$} such that $\rho_i=\pi_i^j \ff{h}_a \ \forall \ i\in \cc{I}_j$. Then $\{\ff{Z}\mr{\ff{h}_j} \ff{X}^j\}_{j\in \cc{J}}$ with $\{\ff{h}_j\Mr{\ff{h}_a} \ff{h}_{j'}\ff{X}^a\}_{j\mr{a} j' \in \cc{J}}$ is a pseudocone over $\ff{X}$: 

\noindent PC0: By Claim 1, in Claim 2 we can take $\ff{r}=id$ and $\varphi=id$, so $\rho_i=id$ and therefore $\ff{h}_{id}=id$.

\noindent PC1: Given $j\mr{a}j'\mr{b}j''\in \cc{J}$ and $i\in \cc{I}_j$, by Claim 1, in Claim 2 we can take $(\ff{r},id)$ representing $\ff{X}^a$, $\ff{X}_{i'}^{j'}\mr{\ff{r}}\ff{X}_i^j$, $(\ff{s},id)$ representing $\ff{X}^b$, $\ff{X}_{i''}^{j''}\mr{\ff{s}}\ff{X}_{i'}^{j'}$ and $(\ff{r}\ff{s},id)$ representing $\ff{X}^{ba}$. Then

$$\vcenter{\xymatrix@C=-0.5pc{  \pi_i^j \deq &&&& \ff{h}_j \ar@{-}[dll] \ar@{-}[dr] \dc{\ff{h}_a} \\
			       \pi_i^j \deq && \ff{X}^a \deq &&& \ff{h}_{j'} \op{\ff{h}_b} \\
			       \pi_i^j && \ff{X}^a && \ff{X}^b && \ff{h}_{j''}   } }
\vcenter{\xymatrix@C=0pc{\ = \ \ } }
\vcenter{\xymatrix@C=-0.5pc{  & \pi_i^j \ar@{-}[dl] && \ar@{}[dl]|{\ff{h}_{(a,\ff{r},id)}} & \ff{h}_j \ar@{-}[dr] \\                               
			      \ff{r} \dl & \dc{=} & \pi_{i'}^{j'} \dr &&& \ff{h}_{j'}  \deq \\
			      \pi_i^j \deq & & \ff{X}^a \deq &&& \ff{h}_{j'} \op{\ff{h}_b} \\
			      \pi_i^j && \ff{X}^a && \ff{X}^b && \ff{h}_{j''} } }
\vcenter{\xymatrix@C=0pc{\ = \ \ } }
\vcenter{\xymatrix@C=-0.5pc{  & \pi_i^j \ar@{-}[dl] && \ar@{}[dl]|{\ff{h}_{(a,\ff{r},id)}} & \ff{h}_j \ar@{-}[dr] \\                               
			      \ff{r} \deq & & \pi_{i'}^{j'} \deq && &\ff{h}_{j'}  \op{\ff{h}_b} \\
			      \ff{r} \dl & \dc{=} & \pi_{i'}^{j'} \dr && \ff{X}^b \deq && \ff{h}_{j''}  \deq \\
			      \pi_i^j && \ff{X}^a && \ff{X}^b && \ff{h}_{j''} } }
\vcenter{\xymatrix@C=0pc{\   = \ \   } }$$			      
$$
\vcenter{\xymatrix@C=-0.5pc{ & \pi_i^j \ar@{-}[dl] && \dc{\ff{h}_{(a,\ff{r},id)}} & \ff{h}_j \ar@{-}[dr] \\
			      \ff{r} \deq &&& \pi_{i'}^{j'} \ar@{-}[dl] & \dc{\ff{h}_{(b,\ff{s},id)}} & \ff{h}_{j'} \ar@{-}[dr] \\
			      \ff{r} \deq && \ff{s} \dl &\dc{=}& \pi_{i''}^{j''} \dr&& \ff{h}_{j''} \deq \\
			      \ff{r} \dl & \dc{=} & \pi_{i'}^{j'} \dr && \ff{X}^b \deq && \ff{h}_{j''}  \deq \\
			      \pi_i^j && \ff{X}^a && \ff{X}^b && \ff{h}_{j''}
			      } }
\vcenter{\xymatrix@C=0pc{\   = \ \ } }
\vcenter{\xymatrix@C=-0.5pc{ & \pi_i^j \ar@{-}[dl] && \dc{\quad \ff{h}_{(ba,\ff{rs},id)}} && \ff{h}_j \ar@{-}[dr] \\
			     \ff{r} \deq && \ff{s} \dl &\dc{=}& \pi_{i''}^{j''} \dr&& \ff{h}_{j''} \deq \\
			      \ff{r} \dl & \dc{=} & \pi_{i'}^{j'} \dr && \ff{X}^b \deq && \ff{h}_{j''}  \deq \\
			      \pi_i^j && \ff{X}^a && \ff{X}^b && \ff{h}_{j''} } }
\vcenter{\xymatrix@C=0pc{\   = \ \ } }
\vcenter{\xymatrix@C=-0.5pc{ \pi_i^j \deq &&&& \ff{h}_j \ar@{-}[dll] \dc{\ff{h}_{ba}} \ar@{-}[drr] \\
			      \pi_i^j && \ff{X}^a && \ff{X}^b && \ff{h}_{j''} } } $$

\noindent where the first, the third and the last equalities hold by definition of $\ff{h}_{(a,\ff{r},id)}$, $\ff{h}_{(b,\ff{s},id)}$ and $\ff{h}_{(ba,\ff{rs},id)}$ respectively; and the fourth equality is due to the fact that $\ff{h}$ is a pseudocone.

Since we checked this for any $i\in \cc{I}_j$, it follows:

$$\vcenter{\xymatrix@C=-0.5pc{&& \ff{h}_j \ar@{-}[dll] \ar@{-}[dr] \dc{\ff{h}_a} \\
			       \ff{X}^a \deq &&& \ff{h}_{j'} \op{\ff{h}_b} \\
			       \ff{X}^a && \ff{X}^b && \ff{h}_{j''}   } }
\vcenter{\xymatrix@C=0pc{ = } }
\vcenter{\xymatrix@C=-0.5pc{&& \ff{h}_j \ar@{-}[dll] \dc{\ff{h}_{ba}} \ar@{-}[drr] \\
			     \ff{X}^a && \ff{X}^b && \ff{h}_{j''} } } $$

\noindent PC2: Given $j\cellrd{a}{\alpha}{b}j'\in \cc{J}$ and $i\in \cc{I}_j$, there is $\ff{X}_{i'}^{j'}\cellrd{\ff{r}}{\theta}{\ff{s}}\ff{X}_i^j$ and appropriate invertible 2-cells $\varphi$, $\psi$ such that $(\theta,\ff{r},\varphi,\ff{s},\psi)$ represents $\ff{X}^{\alpha}$. By Claim 1, in Claim 2 we can take those representatives of $\ff{X}^a$ and $\ff{X}^b$ and then:

\vspace{-2ex}

$$\vcenter{\xymatrix@C=-0.3pc@R=3ex{ \pi_i^j \deq &&& \ff{h}_j \op{\ff{h}_b}  \\
				\pi_i^j && \ff{X}^b && \ff{h}_{j'}    }}
\vcenter{\xymatrix@C=0pc@R=3ex{\ \  = \ \ } }
\vcenter{\xymatrix@C=-0.3pc@R=3ex{   & \pi_i^j \ar@{-}[dl] & \dc{\ff{h}_{(b,\ff{s},\psi)}} &  \ff{h}_j \ar@{-}[dr] \\
			    \ff{s} \dl & \dc{\psi^{-1}} & \pi_{i'}^{j'} \dr && \ff{h}_{j'} \deq \\
			    \pi_i^j  && \ff{X}^b && \ff{h}_{j'}    }}
\vcenter{\xymatrix@C=0pc@R=3ex{\ \  = \ \ } }     
\vcenter{\xymatrix@C=-0.3pc@R=3ex{  & \pi_i^j \ar@{-}[dl] & \dc{\ff{h}_{(a,\ff{r},\varphi)}} &  \ff{h}_j \ar@{-}[dr] \\
			    \ff{r} \dcell{\theta} && \pi_{i'}^{j'} \deq && \ff{h}_{j'} \deq \\ 
			    \ff{s} \dl & \dc{\psi^{-1}} & \pi_{i'}^{j'} \dr && \ff{h}_{j'} \deq \\
			    \pi_i^j  && \ff{X}^b && \ff{h}_{j'}    }}
\vcenter{\xymatrix@C=0pc@R=3ex{\ \  = \ \ } }
$$

\vspace{-3ex}

$$\vcenter{\xymatrix@C=-0.3pc@R=3ex{  & \pi_i^j \ar@{-}[dl] & \dc{\ff{h}_{(a,\ff{r},\varphi)}} &  \ff{h}_j \ar@{-}[dr] \\
			    \ff{r} \dl & \dc{\varphi^{-1}} & \pi_{i'}^{j'} \dr && \ff{h}_{j'} \deq \\ 
			    \pi_i^j \deq && \ff{X}^a \dcell{\ff{X}^{\alpha}} && \ff{h}_{j'} \deq \\ 
			    \pi_i^j  && \ff{X}^b && \ff{h}_{j'}    }}
\vcenter{\xymatrix@C=0pc@R=3ex{\ \  = \ \ } }     
\vcenter{\xymatrix@C=-0.3pc@R=3ex{ \pi_i^j \deq &&& \ff{h}_j \op{\ff{h}_a} \\
			    \pi_i^j \deq && \ff{X}^a \dcell{\ff{X}^{\alpha}} && \ff{h}_{j'} \deq \\ 
			    \pi_i^j  && \ff{X}^b && \ff{h}_{j'}    }}  
$$   
\noindent where the first and last equalities hold by definition of $\ff{h}_{(a,\ff{r},\varphi)}$ and $\ff{h}_{(b,\ff{s},\psi)}$ respectively, the second equality is due to the fact that $\ff{h}$ is a pseudocone and the third one is valid because $(\theta,\ff{r},\varphi,\ff{s},\psi)$ represents $\ff{X}^\alpha$. 
			    
Since we checked this for any $i\in \cc{I}_j$, it follows:
$$
\vcenter{\xymatrix@C=-0.3pc@R=3ex{ && \ff{h}_j \op{\ff{h}_b}  \\
				& \ff{X}^b && \ff{h}_{j'}    }}
\vcenter{\xymatrix@C=0pc@R=3ex{\ \  = \ \ } }
\vcenter{\xymatrix@C=-0.3pc@R=3ex{ &\ff{h}_j  \op{\ff{h}_a} & \\
			     \ff{X}^a \dcell{\ff{X}^{\alpha}} && \ff{h}_{j'} \deq \\ 
			      \ff{X}^b && \ff{h}_{j'}    }}  
$$

\emph{2. $p$ is full and faithful}:
Let $\{\ff{Z} \cellrd{\pi_i^j \ff{h}_j}{\rho_{(i,j)}}{\pi_i^j \ff{m}_j} \ff{X}_i^j \}_{(i,j)\in \cc{K}_\ff{X}}$
be a morphism of pseudocones for $\tilde{\ff{X}}$.
It is easy to check that for each $j\in \cc{J}$, $\{\ff{Z}  \cellrd{\pi_i^j \ff{h}_j}{\rho_{(i,j)}}{\pi_i^j \ff{m}_j} \ff{X}_i^j \}_{i\in \cc{I}_j}$ \hfill is a morphism of pseudocones for $\ff{X}^j$.  Then arguing 

\vspace{-4ex}

\noindent as above, there exists a unique morphism $\ff{Z} \cellrd{\ff{h}_j}{\rho_j}{\ff{m}_j} \ff{X}^j \in 2$-$\cc{P}ro(\cc{C})$ such that $\forall \ i\in \cc{I}_j$, $\pi_i^j\rho_j=\rho_{(i,j)}$. It only remains to prove that $\{\rho_j\}_{j\in \cc{J}}$ is a morphism of pseudocones: 

PCM: Given $j\mr{a}j'\in \cc{J}$ and $i\in \cc{I}_j$, by Claim 1, in Claim 2 we can take $(\ff{r},id)$ representing $\ff{X}^a$, $\ff{X}_{i'}^{j'}\mr{\ff{r}}\ff{X}_{i}^{j}$ and then:

$$
\vcenter
  {
   \xymatrix@C=-0.3pc@R=3ex
      {
       \pi_i^j \deq &&& \ff{h}_j \dcell{\rho_j} 
      \\
	   \pi_i^j \deq &&& \ff{m}_j \op{\ff{m}_a} 
	  \\
	   \pi_i^j && \ff{X}^a && \ff{m}_{j'} 
	  }
  }
\vcenter
  {
   \xymatrix@C=0pc@R=3ex
     {
      \ \  = \ \ 
     } 
  }    
\vcenter
  {
   \xymatrix@C=-0.3pc@R=3ex
     {
        \pi_i^j \dl && \ar@{}[dl]|{\quad \rho_{(i,j)}} 
                                            & \ff{h}_j \dr 
      \\
	    \pi_i^j \deq &&&  \ff{m}_j \op{\ff{m}_a}  
	  \\
	    \pi_i^j && \ff{X}^a && \ff{m}_{j'} 
	  }
  }
\vcenter
  {
   \xymatrix@C=0pc@R=3ex
     {
      \ \  = \ \ 
     } 
  }   
\vcenter
  {
   \xymatrix@C=-0.3pc@R=3ex
     { \pi_i^j \deq &&& \ff{h}_j \op{\ff{h}_a} 
      \\
	   \pi_i^j \dl &\dc{=}& \ff{X}^a \dr&& \ff{h}_{j'} \deq
	  \\
	   \ff{r} \deq && \pi_{i'}^{j'} \dl 
	                & \dc{\rho_{(i',j')}} & \ff{h}_{j'} \dr 
	  \\
	   \ff{r} \dl & \dc{=} & \pi_{i'}^{j'} \dr 
	                                    && \ff{m}_{j'} \deq 
	  \\
	    \pi_i^j \deq & \! \comw{X_a} \! & \ff{X}^a 
	                                    && \ff{m}_{j'} 
	  \\
	    \pi_i^j \deq &&& \cl{\ \! \ff{m}_a^{-1}} \ff{m}_j 
	                                         \op{\ff{m}_a} 
	  \\
	    \pi_i^j && \ff{X}^a && \ff{m}_j 
	  }
  }
\vcenter
  {
   \xymatrix@C=0pc@R=3ex
     {
      \ \  = \ \ 
     } 
  }   
$$

$$\vcenter{\xymatrix@C=0pc@R=3ex{\ \  = \ \ } }   
\vcenter{\xymatrix@C=-0.3pc@R=3ex{ \pi_i^j \deq &&& \ff{h}_j \op{\ff{h}_a} \\
			      \pi_i^j \dl & \dc{=} & \ff{X}^a \dr && \ff{h}_{j'}  \deq\\
			      \ff{r} \deq && \pi_{i'}^{j'} \dl & \dc{\rho_{(i',j')}} & \ff{h}_{j'} \dr \\
			      \ff{r} \dl & \dc{=} & \pi_{i'}^{j'} \dr && \ff{m}_{j'} \deq \\
			      \pi_i^j & \! \comw{X_a} \! & \ff{X}^a && \ff{m}_{j'} }}
\vcenter{\xymatrix@C=0pc@R=3ex{\ \  = \ \ }}   
\vcenter{\xymatrix@C=-0.3pc@R=3ex{ \pi_i^j \deq &&& \ff{h}_j \op{\ff{h}_a} \\
                             \pi_i^j \dl &\dc{=}& \ff{X}^a \dr&& \ff{h}_{j'}  \deq\\  
			      \ff{r} \deq && \pi_{i'}^{j'} \deq && \ff{h}_{j'} \dcell{\rho_{j'}} \\ 
			      \ff{r} \dl& \dc{=}& \pi_{i'}^{j'} \dr&& \ff{m}_{j'} \deq \\
			      \pi_i^j & \! \comw{X_a} \! & \ff{X}^a && \ff{m}_{j'}}}
\vcenter{\xymatrix@C=0pc@R=3ex{\ \  = \ \ }}   
\vcenter{\xymatrix@C=-0.3pc@R=3ex{ \pi_i^j \deq &&& \ff{h}_j \op{\ff{h}_a} \\
                             \pi_i^j \deq && \ff{X}^a \deq && \ff{h}_{j'}  \dcell{\rho_{j'}} \\  			     
			      \pi_i^j  & & \ff{X}^a  && \ff{m}_{j'} }}$$

\noindent where the second equality is valid because $\rho$ is a morphism of pseudocones.
			      
Since we checked this for any $i\in \cc{I}_j$, it follows:                                   

$$\vcenter{\xymatrix@C=-0.3pc@R=3ex{ & \ff{h}_j \dcell{\rho_j} \\
				& \ff{m}_j \op{\ff{m}_a} \\
				 \ff{X}^a && \ff{m}_{j'} }}
\vcenter{\xymatrix@C=0pc{\ \  = \ \ } }  
\vcenter{\xymatrix@C=-0.3pc@R=3ex{ & \ff{h}_j \op{\ff{h}_a} \\
			       \ff{X}^a \deq && \ff{h}_{j'} \dcell{\rho_{j'}}  \\
			       \ff{X}^a && \ff{m}_{j'} }}$$

\end{proof}

\emph{Proof of Claim 1}. 
First assume that $i'=i''$ and $(\ff{r},\varphi)$, $(\ff{s},\psi)$ are related by a \mbox{2-cell} 
$(i,j)\cellrd{(a,\ff{r},\varphi)}{(a,\theta)}{(a,\ff{s},\psi)}(i',j')$ in $\cc{K}_{\ff{X}}$. 
Then:

\vspace{-2ex}

$$\vcenter{\xymatrix@C=-0.5pc{  & \pi_i^j \ar@{-}[dl] && \ar@{}[dl]|{\ff{h}_{(a,\ff{s},\psi)}} & \ff{h}_j \ar@{-}[dr] \\                               
			      \ff{s} \dl & \dc{\psi^{-1}} & \pi_{i''}^{j'} \dr &&& \ff{h}_{j'}  \deq \\
			      \pi_i^j & & \ff{X}^a &&& \ff{h}_{j'} } }
\vcenter{\xymatrix@C=0pc{\ \  = \ \ } } 
\vcenter{\xymatrix@C=-0.5pc{  & \pi_i^j \ar@{-}[dl] && \ar@{}[dl]|{\ff{h}_{(a,\ff{r},\varphi)}} & \ff{h}_j \ar@{-}[dr] \\                               
			      \ff{r} \dcell{\theta} & & \pi_{i''}^{j'} \deq &&& \ff{h}_{j'}  \deq \\
			      \ff{s} \dl & \dc{\psi^{-1}} & \pi_{i''}^{j'} \dr & && \ff{h}_{j'} \deq \\ 
			      \pi_i^j & & \ff{X}^a &&& \ff{h}_{j'} } }
\vcenter{\xymatrix@C=0pc{\ \  = \ \ } } 
\vcenter{\xymatrix@C=-0.5pc{  & \pi_i^j \ar@{-}[dl] && \ar@{}[dl]|{\ff{h}_{(a,\ff{r},\varphi)}} & \ff{h}_j \ar@{-}[dr] \\                               
			      \ff{r} \dl & \dc{\varphi^{-1}} & \pi_{i''}^{j'} \dr &&& \ff{h}_{j'}  \deq \\
			      \pi_i^j & & \ff{X}^a &&& \ff{h}_{j'} } }			      
$$
\noindent where the first equality holds because $\ff{h}$ is a pseudocone, and the second because $\theta$ represents $id$ (the identity of $\ff{X}^a$).

The general case reduces to this one as follows:

\noindent We have 
$
\vcenter
    {
     \xymatrix@R=.01pc
        {
         & (i',j') 
         \\
         (i,j) \ar[ru]^{(a,\ff{r},\varphi)}  \ar[dr]_{(a,\ff{s},\psi)} 
         \\
		 & (i'',j') 
		 }
    }
$.
Take 
$
\vcenter
    {
     \xymatrix@R=.01pc
         {
            i'\ar[rd]^u 
         \\
 		  & k 
		 \\
		    i'' \ar[ru]_v
		  }
    }
$
in 	$\cc{I}_j$. This yields a \mbox{particular} instance of lemma \ref{lema2}:
$$
\xymatrix@C=9ex@R=7ex
        {
         \ff{X}^{j'}   
             \ar@<1.6ex>[r]^{\ff{X}^a}
             \ar@{}@<-1.3ex>[r]^{\!\! {id\,} \, \!\Downarrow}
             \ar@<-1.1ex>[r]_{\ff{X}^a}
             \ar[d]^{\pi_{k}}
         & \ff{X}^j 
             \ar[d]^{\pi_i}
        \\
         \ff{X}_{k}^{j'}
             \ar@<1.6ex>[r]^{\ff{r}\ff{X}_u^{j'}}
             \ar@{}@<-1.3ex>[r]^{}
             \ar@<-1.1ex>[r]_{\ff{s}\ff{X}_v^{j'}}
         & \ff{X}_i^j
        }
$$
with $(\ff{r}\ff{X}_u^{j'},(\ff{r}\pi_u^{j'})\circ \varphi)$ and 
$(\ff{s}\ff{X}_v^{j'},(\ff{s}\pi_v^{j'})\circ \psi)$ 
both representing $\ff{X}^a$. 
It \mbox{follows} there exists 
$k\stackrel{w}{\rightarrow} k'$ and 
$\ff{X}_{k'}^{j'}\cellrd{\ff{r}\ff{X}_u^{j'}\ff{X}_w^{j'}}{\theta}{\ff{s}\ff{X}_v^{j'}\ff{X}_w^{j'}} \ff{X}_i^j$ such that 
\mbox{$(\theta,\; \ff{r}\ff{X}_u^{j'}\ff{X}_w^{j'},\; \ff{r}\ff{X}_u^{j'}\pi_w^{j'}\circ \ff{r}\pi_u^{j'}\circ \varphi,\; \ff{s}\ff{X}_v^{j'}\ff{X}_w^{j'},\;\ff{s}\ff{X}_v^{j'}\pi_w^{j'}\circ \ff{s}\pi_v^{j'}\circ \psi)$} 
represents  $id$ (the identity of $\ff{X}^a$).

\noindent Considering $(\ff{r}\ff{X}_u^{j'}\ff{X}_w^{j'},\;\ff{r}\ff{X}_u^{j'}\pi_w^{j'}\circ \ff{r}\pi_u^{j'}\circ \varphi)$ and 
\mbox{$(\ff{s}\ff{X}_v^{j'}\ff{X}_w^{j'},\;\ff{s}\ff{X}_v^{j'}\pi_w^{j'}\circ \ff{s}\pi_v^{j'}\circ \psi)$}
both representing $\ff{X}^a$, we have a situation that corresponds to the previous case. 
Thus:

\vspace{1ex} 
 
$$\vcenter{\xymatrix@C=-0.5pc{  & \pi_i^j \ar@{-}[dl] &&& \ar@{}[dl]|{\ff{h}_{(a,\ff{r},\varphi)}} && \ff{h}_j \ar@{-}[dr] \\                               
			      \ff{r} \deq &&& \pi_{i'}^{j'} \ar@{-}[dl] && \ar@{}[d]|{\ff{h}_{(j',\ff{X}^{j'}_u\ff{X}^{j'}_w,\ff{X}^{j'}_u \pi^{j'}_w)}} && \ff{h}_{j'} \ar@{-}[dr]   \\
			      \ff{r} \deq && \ff{X}_u^{j'} \deq && \ff{X}_w^{j'} \dl && \pi_{k'}^{j'} \dr && \ff{h}_{j'} \deq \\
			      \ff{r} \deq && \ff{X}_u^{j'} \ar@{-}[dr] &&& \pi_k^{j'} \ar@{-}[dl] \ar@{}[u]|{{\pi_w^{j'}}^{-1}} &&&  \ff{h}_{j'} \deq \\
			      \ff{r} \ar@{-}[dr] &&& \ar@{}[r]|{\displaystyle \pi_{i'}^{j'}} \ar@{}[ur]|{{\pi_u^{j'}}^{-1} } \dr \comw{X} & \comw{|} &&&&  \ff{h}_{j'} \deq  \\
			      & \pi_i^j & \ar@{}[u]|{\varphi^{-1}} &\ff{X}^a &&&&& \ff{h}_{j'}   }}
\vcenter{\xymatrix@C=0pc{\ \  = \ \ } } 
\vcenter{\xymatrix@C=-0.5pc{  & \pi_i^j \ar@{-}[dl] &&& \ar@{}[dl]|{\ff{h}_{(a,\ff{s},\psi)}} && \ff{h}_j \ar@{-}[dr] \\                               
			      \ff{s} \deq &&& \pi_{i''}^{j'} \ar@{-}[dl] && \ar@{}[d]|{\ff{h}_{(j',\ff{X}^{j'}_v\ff{X}^{j'}_w,\ff{X}^{j'}_v \pi^{j'}_w)}} && \ff{h}_{j'} \ar@{-}[dr]   \\
			      \ff{s} \deq && \ff{X}_v^{j'} \deq && \ff{X}_w^{j'} \dl && \pi_{k'}^{j'} \dr && \ff{h}_{j'} \deq \\
			      \ff{s} \deq && \ff{X}_v^{j'} \ar@{-}[dr] &&& \pi_k^{j'} \ar@{-}[dl] \ar@{}[u]|{{\pi_w^{j'}}^{-1}} &&&  \ff{h}_{j'} \deq \\
			      \ff{s} \ar@{-}[dr] &&& \ar@{}[r]|{\displaystyle \pi_{i''}^{j'}} \ar@{}[ur]|{{\pi_v^{j'}}^{-1} } \dr \comw{X} & \comw{|} &&&&  \ff{h}_{j'} \deq  \\
			      & \pi_i^j & \ar@{}[u]|{\psi^{-1}} &\ff{X}^a &&&&& \ff{h}_{j'}   }}$$

\vspace{1ex}

From \ref{*1} and the fact that $\ff{X}^j$ is a pseudocone, it follows that  
$(\ff{r}(\pi_u^{j'})^{-1}\circ \ff{r}\ff{X}_u^{j'}(\pi_w^{j'})^{-1})\ff{h}_{j'}\circ \ff{r}\ff{h}_{(j',\ff{X}_u^{j'}\ff{X}_w^{j'},\ff{X}_
u^{j'}\pi_w^{j'})}$ and 
$(\ff{s}(\pi_v^{j'})^{-1}\circ \ff{s}\ff{X}_v^{j'}(\pi_w^{j'})^{-1})\ff{h}_{j'}\circ \ff{s}\ff{h}_{(j',\ff{X}_v^{j'}\ff{X}_w^{j'},\ff{X}_v^{j'}\pi_w^{j'})}$ are identities. So 
\mbox{$\varphi^{-1}\ff{h}_{j'}\circ \ff{h}_{(a,\ff{r},\varphi)}=\psi^{-1}\ff{h}_{j'}\circ \ff{h}_{(a,\ff{s},\psi)}$} as we wanted to prove.
\qed

\vspace{1ex}

\emph{Proof of Claim 2}.
Given any $i \mr{u} k \in \cc{I}_j$, we have to check the $PCM$ equation in \ref{pseudocone}. Given the pair $(\ff{s}, \psi)$ used to define $\rho_k$, it is possible to choose a pair $(\ff{r}, \varphi)$ to define $\rho_i$ in such a way that the equation holds. This arguing is justified by Claim 1.
\qed

\begin{corollary}\label{cerradaporpseudolimites}
$2$-$\cc{P}ro(\cc{C})$ is closed under small 2-cofiltered \mbox{pseudolimits.} Considering the equivalence in \ref{eqconloscolimderep}, it follows that the inclusion
\mbox{$\cc{H}om(\cc{C},\, \cc{C}at)_{fc} \;\subset\; \cc{H}om(\cc{C},\, \cc{C}at)$} is closed under small 2-filtered \mbox{pseudocolimits} \cqd
\end{corollary}

\vspace{1ex}

 \subsection{Universal property of $\Pro{C}$} $ $
 
In this subsection we prove for \mbox{2-pro-objects} the universal property established for pro-objects in \cite[Ex. I, Prop. 8.7.3.]{G2}.
Consider the 2-functor $\cc{C} \mr{c} 2$-$\cc{P}ro(\cc{C})$ of Corollary \ref{CisPro} and a 2-pro-object $\ff{X} = (\ff{X}_i)_{i\in \cc{I}}$. Given a 2-functor $\cc{C} \mr{\ff{F}} \cc{E}$ into a 2-category closed under
small \mbox{2-cofiltered} pseudolimits, we can naively extend
$\ff{F}$ into a 2-cofiltered pseudolimit preserving 2-functor
$2\hbox{-} \cc{P}ro(\cc{C}) \mr{\widehat{\ff{F}}} \cc{E}$ by defining
$\widehat{\ff{F}}\ff{X} = \Lim{i \in \cc{I}}{\ff{F}\ff{X}_i}$.
This is just part of a 2-equivalence of 2-categories that we develop with the necessary precision in this subsection. First the universal property  should be wholly established for $\cc{E} = \cc{C}at$, and only afterwards can be lifted to any 2-category $\cc{E}$ closed under small 2-cofiltered pseudolimits.

\begin{lemma}\label{casoCat}
 Let $\cc{C}$ be a 2-category and $\ff{F}:\cc{C}\mr{} \cc{C}at$ a \mbox{2-functor.} Then, there exists a 2-functor
 $\widehat{\ff{F}}: 2\hbox{-}\cc{P}ro(\cc{C}) \mr{} \cc{C}at$ that preserves small 2-cofiltered pseudolimits, and an isomorphism
 $\widehat{\ff{F}}c \mr{\cong} \ff{F}$ in $\cc{H}om(\cc{C},\cc{C}at)$.
\end{lemma}

\vspace{1ex}

\begin{proof}
Let $\ff{X}=(\ff{X}_i)_{i\in \cc{I}}\in 2\hbox{-}\cc{P}ro(\cc{C})$ be a 2-pro-object. Define:

\vspace{1ex}

\noindent
$
\widehat{\ff{F}}\ff{X} \;=\;
(\cc{H}om(\cc{C},\cc{C}at)(-,\ff{F})\circ \ff{L})\ff{X} \;=\;
\cc{H}om(\cc{C},\cc{C}at)(\coLim{i\in \cc{I}}{\cc{C}(\ff{X}_i,-)},\ff{F})
\mr{\cong}
$

\hfill
$
\mr{\cong} \Lim{i\in \cc{I}}{\cc{H}om(\cc{C},\cc{C}at)(\cc{C}(\ff{X}_i,-),\ff{F})}
\mr{\cong} \Lim{i\in \cc{I}}{\ff{F}\ff{X}_i}.
$

\vspace{1ex}

Where $\ff{L}$ is the 2-functor of \ref{eqconloscolimderep}, the first isomorphism is by definition of pseudocolimit \ref{colimits}, and the second is the Yoneda isomorphism \ref{2Yoneda}. Since it is a 2-equivalence, the 2-functor $\ff{L}$ preserves any pseudocolimit. Then by Corollary \ref{cerradaporpseudolimites} it follows that the composite
$\cc{H}om(\cc{C},\cc{C}at)(-,\ff{F})\circ \ff{L}$ preserves small 2-cofiltered pseudolimits
 \end{proof}

\begin{theorem} \label{pseudouniversalcat}
Let $\cc{C}$ be any 2-category. Then, pre-composition with
\mbox{$\cc{C} \mr{c} 2\hbox{-}\cc{P}ro(\cc{C})$}  is a 2-equivalence of 2-categories:
$$
\xymatrix@R=.7pc@C=.5pc
 {
  & \cc{H}om(2\hbox{-}\cc{P}ro(\cc{C}),\cc{C}at)_+ \ar[rrrr]^>>>>>>>>{c^*}
  & & & & \comw{\ff{X}\ff{X}} \cc{H}om(\cc{C},\cc{C}at) \comw{\ff{X} \ff{X}}
 }
$$
(where $\cc{H}om(2\hbox{-}\cc{P}ro(\cc{C}),\cc{C}at)_+$ stands for the full subcategory whose objects are those 2-functors that preserve small 2-cofiltered pseudolimits).
\end{theorem}
\begin{proof}
We check that the 2-functor $c^*$ is essentially surjective on objects and 2-fully-faithful:

\emph{Essentially surjective on objects}: It follows from lemma \ref{casoCat}.

\emph{2-fully-faithful}: We check that if $\ff{F}$ and $\ff{G}$ are 2-functors from $2\hbox{-}\cc{P}ro(\cc{C})$ to $\cc{C}at$ that preserve small 2-cofiltered pseudolimits, then
\begin{equation} \label{hrest}
\cc{H}om(2\hbox{-}\cc{P}ro(\cc{C}),\cc{C}at)_+(\ff{F},\ff{G})
                    \mr{c^*} \cc{H}om(\cc{C},\cc{C}at)(\ff{F}c, \ff{G}c)
\end{equation}
is an isomorphism of categories.

Let $\ff{F}c \cellrd{\theta c}{\mu}{\eta c} \ff{G}c \in \cc{H}om(\cc{C},\cc{C}at)(\ff{F}c, \ff{G}c)$.
It can be easily checked that the composites
$\{\ff{F}\ff{X} \mr{\ff{F} \pi_i} \ff{F}\ff{X}_i
\cellrd{\theta_{\ff{X}_i}}{\mu_{\ff{X}_i}}{\eta_{\ff{X}_i}} \ff{G}\ff{X}_i\}_{i\in \cc{I}}$
 determine two pseudocones for $\ff{G}\ff{X}$ together with a morphism of pseudocones. Since $\ff{G}$ preserves small \mbox{2-cofiltered} pseudolimits, post-composing with
 $\ff{GX} \mr{\ff{G}\pi_i} \ff{GX}_i$ is an isomorphism of categories
 $\cc{C}at(\ff{FX}, \ff{GX}) \mr{(\ff{G} \pi)_*}
 \ff{PC}_{\cc{C}at}(\ff{FX}, \ff{GX})$.
 It follows there \mbox{exists} a unique 2-cell in $\cc{C}at$,
 $\ff{F}\ff{X}
\cellrd{\theta'_{\ff{X}}}{\mu'_{\ff{X}}}{\eta'_{\ff{X}}} \ff{G}\ff{X}$,
such that
$\ff{G}\pi_i \theta'_{\ff{X}} \,=\, \theta_{\ff{X}_i} \ff{F} \pi_i$,
$\ff{G}\pi_i \eta'_{\ff{X}} \,=\, \eta_{\ff{X}_i} \ff{F} \pi_i$, and
$\ff{G}\pi_i \mu'_{\ff{X}} \,=\, \mu_{\ff{X}_i} \ff{F} \pi_i$,
$\forall i\in \cc{I}$. It is not difficult to check that $\theta'_{\ff{X}}$, $\eta'_{\ff{X}}$ are in fact 2-natural on $\ff{X}$, and that $\mu'_{\ff{X}}$ is a modification. Clearly $\theta' c = \theta$, $\eta' c = \eta$, and
$\mu' c = \mu$. Thus \ref{hrest} is an isomorphism of categories.
\end{proof}
\begin{lemma}\label{es-sury}
 Let $\cc{C}$ be a 2-category, $\cc{E}$ a 2-category closed under
 small \mbox{2-cofiltered} pseudolimits and $\ff{F}: \cc{C} \mr{} \cc{E}$ a 2-functor. Then, there exists a 2-functor
 \mbox{$\widehat{\ff{F}}: 2\hbox{-}\cc{P}ro(\cc{C}) \mr{} \cc{E}$}
 that preserves small 2-cofiltered pseudolimits, and an isomorphism
 $\widehat{\ff{F}}c \mr{\cong} \ff{F}$ in $\cc{H}om(\cc{C},\cc{E})$.
 \end{lemma}

\begin{proof}
If $\ff{X}=(\ff{X}_i)_{i\in \cc{I}}\in 2\hbox{-}\cc{P}ro(\cc{C})$, define
$\widehat{\ff{F}}\ff{X}=\Lim{i\in \cc{I}^{op}}{\ff{F}\ff{X}_i}$. We will prove that this is the object function part of a 2-functor, and that this 2-functor has the rest of the properties asserted in the proposition.

Consider the composition
$\ff{y}_{(-)} \,\ff{F}:\,\cc{C}\mr{\ff{F}} \cc{E}\mr{\ff{y}_{(-)}}
 \cc{H}om(\cc{E}^{op},\cc{C}at)$,
 where $\ff{y}_{(-)}$ is the Yoneda 2-functor (\ref{Y2functor}). Under the isomorphism \ref{cartesianclosed} this corresponds to a 2-functor
$\cc{E}^{op} \mr{} \cc{H}om(\cc{C},\cc{C}at)$. Composing this
\mbox{2-functor} with a quasi-inverse  $\widehat{(-)}$ for the 2-equivalence in \ref{pseudouniversalcat}, we obtain a 2-functor
$\cc{E}^{op} \mr{} \cc{H}om(2\hbox{-}\cc{P}ro(\cc{C}),\cc{C}at)_+$,
which in turn corresponds to a 2-functor
$ 2\hbox{-}\cc{P}ro(\cc{C}) \mr{\widetilde{\ff{F}}} \cc{H}om(\cc{E}^{op},\cc{C}at)$. The 2-functor $\widetilde{\ff{F}}$ preserves small 2-cofiltered pseudolimits because they are computed pointwise in
$\cc{H}om(\cc{E}^{op},\cc{C}at)$ (\ref{pointwisebilimit}). Chasing the isomorphisms shows that we have the following diagram:
\begin{equation} \label{efetilde}
\xymatrix@C=0.5pc@R=0.5pc
           {
            { }
            & &  { }
           \\
            & {\widetilde{\ff{F}} c \mr{\cong} \ff{y}_{(-)} \ff{F}},
            &
           \\ { }
            & &  { }
           }
\hspace{6ex}
\xymatrix@C=0.5pc@R=0.5pc
           {
            \cc{C} \ar[rr]^<<<<<<<{c}  \ar[dd]_{\ff{F}}
            & &  2\hbox{-}\cc{P}ro(\cc{C}) \ar[dd]^{\widetilde{\ff{F}}}
           \\
            & \Downarrow \cong
            &
           \\ \cc{E} \ar[rr]_<<<<<{\ff{y}_{(-)}}
            & &  \cc{H}om(\cc{E}^{op},\cc{C}at)
           }
\end{equation}
Consider the following chain of isomorphisms (the first and the third because $\widetilde{\ff{F}}$ and $\ff{y}_{(-)}$ preserve pseudolimits (\ref{yonedapreserves}), and the middle one given by \ref{efetilde}):
$$
\widetilde{\ff{F}}\ff{X} =
\widetilde{\ff{F}} \Lim{i \in \cc{I}}{\ff{X}_i} \mr{\cong}
\Lim{i \in \cc{I}}{\widetilde{\ff{F}} c \ff{X}_i} \mr{\cong}
\Lim{i \in \cc{I}}{\ff{y}_{(-)} \ff{F}\ff{X}_i} \ml{\cong}
\ff{y}_{(-)} \Lim{i \in \cc{I}}{\ff{F} \ff{X}_i}.
$$
This shows that $\widetilde{\ff{F}}\ff{X}$ is in the essential image of
$\ff{y}_{(-)}$. Since $\ff{y}_{(-)}$ is \mbox{2-fully} faithful (\ref{yonedaff}), it follows there is a factorization
\mbox{$\ff{y}_{(-)} \widehat{\ff{F}} \mr{\cong} \widetilde{\ff{F}}$,} given by a \mbox{2-functor}
$2\hbox{-}\cc{P}ro(\cc{C}) \mr{\widehat{\ff{F}}} \ff{E}$. Clearly
$\widehat{\ff{F}}$ preserves small \mbox{2-cofiltered} \mbox{pseudolimits}. We have
$
\ff{y}_{(-)} \widehat{\ff{F}} c  \mr{\cong}
\widetilde{\ff{F}} c \mr{\cong}
\ff{y}_{(-)} \widehat{\ff{F}}.
$
Finally, the fully faithfulness of $\ff{y}_{(-)}$ provides an isomorphism
$\ff{y}_{(-)} \widehat{\ff{F}} \mr{\cong} \ff{F}$. This finishes the proof.
\end{proof}
Exactly the same proof of theorem \ref{pseudouniversalcat} applies with an arbitrary \mbox{2-category $\cc{E}$} in place of $\cc{C}at$, and we have:
\begin{theorem} \label{pseudouniversal}
Let $\cc{C}$ be any 2-category, and $\cc{E}$ a 2-category closed under small 2-cofiltered pseudolimits. Then, pre-composition with
\mbox{$\cc{C} \mr{c} 2\hbox{-}\cc{P}ro(\cc{C})$}  is a 2-equivalence of 2-categories:
$$
\xymatrix@R=.7pc@C=.3pc
 {
  & \cc{H}om(2\hbox{-}\cc{P}ro(\cc{C}),\cc{E})_+ \ar[rrrr]^>>>>>>{c^*}
  & & & & \comw{\ff{X}\ff{X}} \cc{H}om(\cc{C},\cc{E}) \comw{\ff{X} \ff{X}}
 }
$$
Where $\cc{H}om(2\hbox{-}\cc{P}ro(\cc{C}),\cc{E})_+$ stands for the full subcategory whose objects are those 2-functors that preserve small 2-cofiltered pseudolimits.
\cqd
\end{theorem}

From theorem \ref{pseudouniversal} it follows automatically the pseudo-functoriality of the assignment of the 2-category $2\hbox{-}\cc{P}ro(\cc{C})$ to each 2-category $\cc{C}$, and in such a way that $c$ becomes a pseudonatural transformation. But we can do better:

\vspace{1ex}

If we put
$\cc{E} = \Pro{D}$ in \ref{pseudouniversal} it follows there is a 2-functor \mbox{(post-composing} with $c$ followed by a quasi-inverse in \ref{pseudouniversal})
\begin{equation} \label{ufa!}
\xymatrix@R=.7pc@C=.3pc
 {
  &  \cc{H}om(\cc{C},\cc{D})
\ar[rrrr]^>>>>>>{\widehat{(-)}}
  & & & & \comw{\ff{X}\ff{X}}
  \cc{H}om(2\hbox{-}\cc{P}ro(\cc{C}),\,
               2\hbox{-}\cc{P}ro(\cc{D}))_+
           \comw{\ff{X} \ff{X}} \hspace{-2.5ex},
 }
\end{equation}
and for each 2-functor $\cc{C} \mr{\ff{F}} \cc{D}$, a diagram:
\begin{equation} \label{pseudosquare}
\xymatrix@C=0.5pc@R=0.5pc
           {
            \Pro{C} \ar[rr]^{\widehat{\ff{F}}}
            & &  \Pro{D}
            \\
            & \Downarrow \cong
            &
            \\
            \cc{C} \ar[rr]_{\ff{\ff{F}}} \ar[uu]_{c}
            & & \cc{D} \ar[uu]_{c}
           }
\end{equation}
Given any 2-pro-object $\ff{X} \in \Pro{C}$, set
$2\hbox{-}\cc{P}ro(\ff{F})(\ff{X}) = \widehat{\ff{F}}\ff{X}$. It is straightforward to check that this determines a 2-functor
$$
\xymatrix@C=5ex@R=0.5pc
           {
            \Pro{C} \ar[rr]^
            {2\hbox{-}\cc{P}ro(\ff{F})}
            & &  \Pro{D}
           }
$$
making diagram \ref{pseudosquare} commutative. It follows we have an \mbox{isomorphism}
\mbox{$\widehat{\ff{\ff{F}}}\ff{X} \mr{\cong}
                   2\hbox{-}\cc{P}ro(\ff{F})(\ff{X})$} 2-natural in $\ff{X}$. This shows that the 2-functor $2\hbox{-}\cc{P}ro(\ff{F})$ preserves small 2-cofiltered pseudolimits because
$\widehat{\ff{F}}$ does. Also, it follows that
$2\hbox{-}\cc{P}ro(\ff{F})$ determines a 2-functor as in \ref{ufa!}. In conclusion, denoting now by $2\hbox{-}\cc{CAT}$ the 2-category of locally small 2-categories (see \ref{2CAT}) we have:
\begin{theorem}
The definition $2\hbox{-}\cc{P}ro(\ff{F})(\ff{X}) = \widehat{\ff{F}}\ff{X}$ determines a \mbox{2-functor}
$$
2\hbox{-}\cc{P}ro(\hbox{-}): 2\hbox{-}\cc{C}\cc{A}\cc{T} \mr{} 2\hbox{-}\cc{C}\cc{A}\cc{T}_+
$$
in such a way that $c$ becomes a 2-natural transformation (where $2\hbox{-}\cc{C}\cc{A}\cc{T}_+$ is the full sub 2-category of locally small 2-categories closed under small \mbox{2-cofiltered} pseudo limits and small pseudolimit preserving 2-functors). \cqd
\end{theorem}

\vspace{2ex}

\vspace{4ex}



\vspace{1ex}

 Mar\'ia Emilia Descotte, 
 
 edescotte@gmail.com
 
 Departamento de Matem\'atica,
 
 F. C. E. y N., Universidad de Buenos Aires,
 
 1428 Buenos Aires, Argentina.
 
 \vspace{2ex}
 
 Eduardo J. Dubuc, 
 
 edubuc@dm.uba.ar

 Departamento de Matem\'atica,
 
 F. C. E. y N., Universidad de Buenos Aires,
 
 1428 Buenos Aires, Argentina.

\end{document}